\documentclass[12pt,reqno]{amsart}
\usepackage{amsmath, amssymb, amsthm, amsfonts}

\usepackage{bbm}
\usepackage{mathrsfs} \usepackage{upgreek}
\usepackage{amscd}
\usepackage{graphicx}
\usepackage[all]{xy}
\usepackage{cite}
\usepackage{tikz}
 \usepackage{color} \usepackage{tikz-cd}
 \usepackage{hyperref}
\usepackage[english]{babel}
\usepackage[margin=1.0in]{geometry}
\usepackage{amsmath, amssymb, amsthm, amsfonts}

\usepackage{mathrsfs} \usepackage{upgreek}
\usepackage{amscd}
\usepackage{graphicx}
\usepackage[all]{xy}
\usepackage{cite}
\usepackage{tikz}
 \usepackage{color} \usepackage{tikz-cd}
 \usepackage{hyperref}
\usepackage[english]{babel}
\usepackage[margin=1.0in]{geometry}

\DeclareMathOperator{\Aut}{Aut}
\DeclareMathOperator{\St}{St}
\DeclareMathOperator{\ev}{ev}
\DeclareMathOperator{\Hom}{Hom}
\DeclareMathOperator{\vir}{vir}

\newtheorem{theorem}{Theorem}[section]
\newtheorem{corollary}[theorem]{Corollary}
\newtheorem{lemma}[theorem]{Lemma}

\newtheorem{proposition}[theorem]{Proposition}

\newtheorem{conjecture}[theorem]{Conjecture}

\begin{document}

\title{Hodge integrals and $\lambda_{g}$ conjecture with target varieties}
\author{Xin Wang}
\address{Department of Mathematics\\ Shandong University \\ Jinan, China}

\email{xinwmath@gmail.com}
\begin{abstract}
In this paper, we propose $\lambda_{g}$ conjecture for Hodge integrals with target varieties. Then we establish relations between Virasoro conjecture and $\lambda_{g}$ conjecture, in particular, we prove $\lambda_{g}$ conjecture in all genus  for smooth projective varieties with semisimple quantum cohomology or smooth algebraic curves. Meanwhile, we also prove $\lambda_{g}$ conjecture in genus zero for any smooth projective varieties. In the end, together with DR formula for $\lambda_g$ class, we obtain a new type of universal constraints for descendant Gromov-Witten invariants. As an application, we prove  $\lambda_{g}$ conjecture in genus one for any smooth projective varieties. 
\end{abstract}

\maketitle
\tableofcontents

\allowdisplaybreaks

\section{Introduction}\label{sec:intro}
\subsection{$\lambda_{g}$ conjecture/theorem}
Let $\overline{\mathcal{M}}_{g,n}$ be the moduli space of stable curves of genus-$g$ with $n$-marked points and $\psi_i$ be the first Chern class of the cotangent space over $\overline{\mathcal{M}}_{g,n}$ at the $i$-th
 marked point. The
Hodge bundle $\mathbb{E}$ over $\overline{\mathcal{M}}_{g,n}$ is the rank $g$ vector bundle with fiber $H^0(C,\omega_C)$ over the domain curve $(C;x_1,...,x_n)$, where $\omega_C$ is the dualizing sheaf of curve $C$. Let $\lambda_{k}$ be
the $k$-th Chern class of $\mathbb{E}$.
The famous Witten’s conjecture predicts 
 the generating function of the $\psi$ class integrals over $\overline{\mathcal{M}}_{g,n}$ satisfies the KdV hierarchy (\cite{witten1990two}). Together with string equation, the KdV hierarchy determine all  $\psi$ class integrals  in all genera. Witten's conjecture has been proven firstly by
 Kontsevich \cite{kontsevich1992intersection}. Besides Witten's conjecture, Getzler and Pandharipande proposed $\lambda_{g}$ conjecture as a consequence of  Virasoro conjecture for $\mathbb{P}^1$, which predicts a simple form of $\psi$ class integrals coupled with $\lambda_{g}$ over $\overline{\mathcal{M}}_{g,n}$ (cf.\cite{getzler1998virasoro}). Then $\lambda_{g}$ conjecture was proved by Faber and Pandharipande via Gromov-Witten theory of $\mathbb{P}^1$ (cf. \cite{faber2003hodge}). Later on, Goulden,  Jackson and  Vakil reproved $\lambda_g$  conjecture using ELSV formula (cf.\cite{goulden2009short}). It is well known that $\lambda_{g}$ conjecture/theorem plays an important role in computing Hodge intergrals and Gromov-Witten invariants.
 
  It is well known that Witten's conjecture for the $\psi$ integrals over $\overline{\mathcal{M}}_{g,n}$ was generalized
 to Virasoro conjecture, which gives universal constraints for tautological integrals over the moduli space of stable maps to
 arbitrary smooth projective varieties. Virasoro conjecture was proposed  by  Eguchi, Hori,
 and Xiong \cite{eguchi1998quantum} and modified by Katz \cite{cox1999mirror}.  Despite many progresses have been made in the literature (cf.  \cite{liu1998virasoro},  \cite{dubrovin1999frobenius}, \cite{getzler1999virasoro}, \cite{liu2001elliptic}, \cite{givental2001semisimple}, \cite{okounkov2006virasoro},  \cite{teleman2012structure}), it is still widely open for general target varieties.  Parallelly, it is very natural to wonder whether $\lambda_{g}$ conjecture/theorem can be generalized to constrain tautological integrals coupled with $\lambda_{g}$ class over the moduli space of stable maps to
 arbitrary smooth projective varieties? If so, what are these constraints and their relationship with Virasoro conjecture?
 In this paper, we will construct these universal constraints explicitly and study their relationship with Virasoro conjecture. 
\subsection{The $\lambda_{g}$ conjecture with target varieties}
Let $X$ be a smooth projective variety  of complex dimension $d$ and $\{\phi_\alpha: \alpha=1,\dots,N\}$ be a basis of its cohomology ring $H^*(X;\mathbb{C})$ with $\phi_1=\mathbf{1}$ the identity and   $\phi_\alpha\in H^{p_\alpha,q_\alpha}(X;\mathbb{C})$
for every $\alpha$.
Recall that the \emph{big phase space} for Gromov--Witten invariants of $X$ is defined to be $\prod_{n=0}^{\infty}H^*(X;\mathbb{C})$ with standard
basis $\{\tau_{n}(\phi_\alpha): \alpha=1,\dots,N, n\geq0\}$.
Denote the coordinates on the big phase space
with respect to the standard basis by  $\{t_n^\alpha\}$.
% Let $F_g^{\mathbb{E}}(\mathbf{t},\mathbf{s})$ be the generating function of genus-$g$ Hodge integrals over the  moduli
% spaces of stable maps with target $X$. Precise definition of $F_g^{\mathbb{E}}(\mathbf{t},\mathbf{s})$ will be given in
% equation~\eqref{eqn:Fg-Hodge-t-s}. Set $\mathcal{D}^{\mathbb{E}}(\mathbf{t},\mathbf{s})=\exp(\sum_{g\geq0}\hbar^{2g-2}F_g^{\mathbb{E}}(\mathbf{t},\mathbf{s}))$ be the all genus total potential function of Hodge integrals, where $\hbar$ is a formal parameter. 
Let $F_g$ be the genus-$g$ generating function, which is a formal power
series of $\mathbf{t}=(t_n^\alpha)$  with coefficients being the  genus-$g$ Gromov-Witten invariants.
Denote $\langle\langle\tau_{n_1}(\phi_{\alpha_1}),\dots,\tau_{n_k}(\phi_{\alpha_k})\rangle\rangle_g$ for the derivatives
of $F_g$ with respect to the variables $t_{n_1}^{\alpha_1},\dots,t_{n_k}^{\alpha_k}$.
Similarly, we use $\langle\langle\tau_{n_1}(\phi_{\alpha_1}),\dots,\tau_{n_k}(\phi_{\alpha_k});\lambda_{g}\rangle\rangle_g$ to denote  the correlation function of Hodge integrals with $\lambda_{g}$ class.
For convenience, we identify $\tau_n(\phi_\alpha)$ with
the coordinate vector field $\frac{\partial}{\partial t_n^\alpha}$
on $\prod_{n=0}^{\infty}H^*(X;\mathbb{C})$ for $n\geq0$. If $n<0$, $\tau_n(\phi_\alpha)$ is understood to be the $0$
vector field. We also abbreviate $\tau_0(\phi_\alpha)$ by $\phi_\alpha$. Let $\phi^\alpha=\sum_{\beta}\eta^{\alpha\beta}\phi_\beta$  with $(\eta^{\alpha\beta})$
 representing the inverse matrix of the Poincar\'e intersection pairing on $H^*(X;\mathbb{C})$. 

%    Let $\mathcal{C}=(\mathcal{C}_\alpha^\beta)$ be the matrix of
% multiplication by the first Chern class $c_1(X)$ in the ordinary cohomology ring, i.e.
% $c_1(X)\cdot \phi_\alpha=\sum_{\beta}\mathcal{C}_\alpha^\beta\phi_\beta$. Here $\mathcal{C}^j$
% is the $j$-th power of the matrix $\mathcal{C}$, $(\mathcal{C}^j)^{\alpha\beta}$ are entries of the matrix $(\mathcal{C}^{j})\eta^{-1}$ and $(\mathcal{C}^{n+1})_{\alpha\beta}$ are entries of the matrix $(\mathcal{C}^{n+1})\eta$. 
% For $n\geq-1$, 
% with  is a change of coordinates. 

Motivated by the famous Virasoro conjecture for Gromov-Witten invariants, we propose the following generalized $\lambda_{g}$ conjecture with target varieties. Denote $b_\alpha=p_\alpha-\frac{1}{2}(d-1)$, $b^\alpha=1-b_\alpha$ and $\Tilde{t}_r^\alpha=t_r^\alpha-\delta_{r}^{1}\delta_{\alpha}^1$, then for any $g,m\geq0, n\geq-1$ and $\beta$, we
define \begin{align*}
&\Theta_{g,n,m,\beta}
\nonumber\\= 
&\sum_{j=0}^{n+1}\sum_{r,\alpha}e_{n+1-j}(r+b_{\alpha}-\frac{1}{2},...,r+n+b_{\alpha}-\frac{1}{2})\tilde{t}_r^{\alpha} 
\langle\langle\tau_{r+n-j}(c_1(X)^{j}\cup\phi_\alpha)\tau_{m}(\phi_\beta);\lambda_{g}\rangle\rangle_{g}
\nonumber\\&+
\sum_{j=0}^{n+1}e_{n+1-j}(m+b_{\beta}+\frac{1}{2},...,m+n+b_{\beta}+\frac{1}{2})\langle\langle \tau_{m+n-j}(c_1(X)^{j}\cup\phi_{\beta});\lambda_{g}\rangle\rangle_{g}
\nonumber\\&+\sum_{j=0}^{n+1}\sum_{s,\alpha}(-1)^{-s-1}e_{n+1-j}(-s+b^{\alpha}-\frac{3}{2},...,-s+n+b^{\alpha}-\frac{3}{2})
\nonumber\\&\quad\quad\cdot\sum_{g_1+g_2=g}\langle\langle\tau_{m}(\phi_\beta)\tau_{-s-1+n-j}(c_1(X)^j\cup\phi^\alpha);\lambda_{g_1}\rangle\rangle_{g_1}
\langle\langle\tau_{s}(\phi_\alpha);\lambda_{g_2}\rangle\rangle_{g_2}
\nonumber\\&+\delta_{g}^0\delta_{m}^{0}\sum_{\alpha}t_0^{\alpha}\int_{X}c_1(X)^{n+1}\cup\phi_\alpha\cup\phi_\beta.
\end{align*}
% \begin{align*}
% \Phi_{g,n}=&\sum_{j=0}^{n+1}\sum_{r,\alpha}e_{n+1-j}(r+b_{\alpha}-\frac{1}{2},...,r+n+b_{\alpha}-\frac{1}{2})\tilde{t}_r^{\alpha} 
% \langle\langle\tau_{r+n-j}(c_1(X)^{j}\cdot\phi_\alpha);\lambda_{g}\rangle\rangle_{g}
% \\&+\frac{1}{2}\sum_{j=0}^{n+1}\sum_{s,\alpha}(-1)^{-s-1}e_{n+1-j}(-s+b^{\alpha}-\frac{3}{2},...,-s+n+b^{\alpha}-\frac{3}{2})
% \\&\quad\quad\cdot\sum_{g_1+g_2=g}\langle\langle\tau_{-s-1+n-j}(c_1(X)^j\cdot\phi^\alpha);\lambda_{g_1}\rangle\rangle_{g_1}
% \langle\langle\tau_{s}(\phi_\alpha);\lambda_{g_2}\rangle\rangle_{g_2}  
% \\&+\frac{1}{2}\delta_{g}^{0}\sum_{\alpha,\beta}(\mathcal{C}^{n+1})_{\alpha\beta}t_0^\alpha t_0^\beta.
% \end{align*}
\begin{conjecture}[$\lambda_{g}$ conjecture]\label{conj:lambda-g}
For any $m$ and 
$\beta$, the genus-$g$ Hodge integrals of $X$  with $\lambda_{g}$ insertion satisfy the following
 identity \begin{align*}
\Theta_{g,n,m,\beta}=0
\end{align*}
for any $n\geq-1$. 
% \begin{align*}
% \frac{\partial}{\partial t_m^\beta}\Phi_{g,n}+\left(\Delta_{1}\hat{e}_{n+1-j}(n,\beta)\right)(m-\frac{1}{2})\langle\langle\tau_{m+n-j}(c_1^j\cdot\phi_\beta);\lambda_{g}\rangle\rangle_{g}=0.    
% \end{align*}
    
\end{conjecture}
% For $n=-1$, it is string equation. 

Despite Virasoro conjecture is widely open in general, many special cases have been proved such as projective varieties with semisimple quantum cohomology and smooth algebraic curves (cf. \cite{givental2001semisimple}, \cite{teleman2012structure}, \cite{okounkov2006virasoro}). For $\lambda_{g}$ conjecture, we can prove 
\begin{theorem}\label{thm:virasoro-lambda-g-conj-rela} Assume Virasoro conjecture for Gromov-Witten invariants of $X$ holds, such as smooth projective varieties with semisimple quantum cohomology or smooth algebraic curves, then Conjecture~\ref{conj:lambda-g} holds.
\end{theorem}
It is well known that 
Virasoro conjecture for genus-0 Gromov-Witten invariants of any target varieties always hold true (cf.\cite{liu1998virasoro}, \cite{dubrovin1999frobenius}, \cite{getzler1999virasoro}, \cite{givental2004symplectic}). For $\lambda_{g}$ conjecture, we can prove
\begin{theorem}\label{thm:lambda-g-conj-g=0}
For any smooth projective variety, the $\lambda_{g}$ conjecture holds for genus-0 Hodge integrals of $X$  with $\lambda_{g}$ insertion.  
\end{theorem}

\subsection{Givental type formula for CohFT associated to  $\lambda_{g}$ invariants}
Let $\Omega^{t}=\{\Omega^t_{g,n}:H^*(X;\mathbb{C})^{\otimes n}\rightarrow H^*(\overline{\mathcal{M}}_{g,n};\mathbb{C})\}_{2g-2+n>0}$ be the cohomological field theory associated to Gromov-Witten invariants of $X$ with base point $t\in H^*(X;\mathbb{C})$. Motivated by Givental-Teleman classification theorem for semisimple cohomological field theories (cf. \cite{givental2001semisimple}, \cite{teleman2012structure}), we have
\begin{theorem}[Theorem~\ref{thm:lambda-g-times-CohFT}]
\label{thm:givental-type-formula-lambda-g}
If $t\in H^*(X;\mathbb{C})$ is a semisimple point of the  CohFT associated to Gromov-Witten theory of $X$, there exist $R$ matrix, such that   
for any vector $w_1,..,w_n\in H^*(X;\mathbb{C})$,
 \begin{align*}
\lambda_{g}\cdot\Omega^{t}_{g,n}(w_1,...,w_n)=\sum_{\Gamma\in G_{g,n}^{c}}\frac{1}{|\Aut(\Gamma)|}\xi_{\Gamma,*}\left(\prod_{v}\sum_{k=0}^{\infty}\frac{1}{k!}\pi_{k*}\left(\lambda_{g(v)}\cdot\omega^{t}_{g(v),n(v)+k}(\cdots)\right)\right)    
 \end{align*}
 where $G_{g,n}^{c}\subset G_{g,n}$ is the set of  stable graphs which is a tree.  $\omega^{t}$ is the topological part of $\Omega^{t}$. All the insertions $(...)$ are the same as in equation~\eqref{eqn:tele-recontr}.
 \end{theorem}
Let $\mathcal{A}_{t}^{\lambda}(\hbar,\mathbf{t})$ be the total ancestor potential of $\lambda_{g}$ invariants. As an application of the  reconstruction theorem~\ref{thm:givental-type-formula-lambda-g}, we obtain  an explicit graph sum formula of $\mathcal{A}_{t}^{\lambda}(\hbar,\mathbf{t})$ in terms of descendant $\lambda_{g}$ invariants over moduli space of stable curves via an $R$ matrix action. The precise formula will be given in Corollary~\ref{cor:total-ancestor-potential-lambda-g} in section~\ref{subsec:recons-lambda-g-invariants}. 
\subsection{New types of universal constraints for descendant Gromov-Witten invariants}
In \cite{janda2017double}, a remarkable formula for the top Chern class of Hodge bundle
$\lambda_g$ was proposed, which is supported on the boundary divisor of curves with
a non-separating node.
That is, in $R^{g}(\overline{\mathcal{M}}_{g,n})$, we have
\begin{align*}
\lambda_g=\frac{(-1)^g}{2^g} \mathsf{P}_{g}^{g}(0,\dots,0) .
\end{align*}

Together this remarkable formula for $\lambda_g$ and Theorem~\ref{thm:virasoro-lambda-g-conj-rela}, we obtain
a new class of universal constraints for pure descendant Gromov-Witten invariants. 
For $g\geq0, n\geq-1$,
we define
\begin{align}\label{eqn:theta-Pixton-class-gnmbeta}
&\Theta^{\mathsf{P}}_{g,n,m,\beta}
\nonumber\\=    
&\sum_{j=0}^{n+1}\sum_{r,\alpha}\sum_{\Gamma\in G_{g,2}}e_{n+1-j}(r+b_{\alpha}-\frac{1}{2},...,r+n+b_{\alpha}-\frac{1}{2})\tilde{t}_r^{\alpha} 
\langle\langle\tau_{r+n-j}(c_1(X)^{j}\cup\phi_\alpha)\tau_{m}(\phi_\beta);\mathsf{P}^g_\Gamma(0,0)\rangle\rangle_{\Gamma}
\nonumber\\&+
\sum_{j=0}^{n+1}\sum_{\Gamma\in G_{g,1}}e_{n+1-j}(m+b_{\beta}+\frac{1}{2},...,m+n+b_{\beta}+\frac{1}{2})\langle\langle \tau_{m+n-j}(c_1(X)^{j}\cup\phi_{\beta});\mathsf{P}^g_\Gamma(0)\rangle\rangle_{\Gamma}
\nonumber\\&+\sum_{j=0}^{n+1}\sum_{s,\alpha}(-1)^{-s-1}e_{n+1-j}(-s+b^{\alpha}-\frac{3}{2},...,-s+n+b^{\alpha}-\frac{3}{2})
\nonumber\\&\quad\quad\cdot\sum_{g_1+g_2=g}\sum_{\substack{\Gamma_1\in G_{g_1,2}\\ \Gamma_2\in G_{g_2,1}}}\langle\langle\tau_{m}(\phi_\beta)\tau_{-s-1+n-j}(c_1(X)^j\cup\phi^\alpha);\mathsf{P}^{g_1}_{\Gamma_1}(0,0)\rangle\rangle_{\Gamma_1}
\langle\langle\tau_{s}(\phi_\alpha);\mathsf{P}^{g_2}_{\Gamma_2}(0)\rangle\rangle_{\Gamma_2}
\nonumber\\&+\delta_{g}^0\delta_{m}^{0}\sum_{\alpha}t_0^{\alpha}\int_{X}c_1(X)^{n+1}\cup\phi_\alpha\cup\phi_\beta
\end{align}
where the correlation functions $\langle\langle\mathcal{W}_1\ldots\mathcal{W}_k;\mathsf{P}_{\Gamma}^g(0,...,0)\rangle\rangle_{\Gamma}$ is defined in subsection~\ref{subsec:proof-of-thm1.5}. 
% \begin{align*}
% \Phi_{g,n}^{\mathsf{P}}=&\sum_{j=0}^{n+1}\sum_{r,\alpha}e_{n+1-j}(r+b_{\alpha}-\frac{1}{2},...,r+n+b_{\alpha}-\frac{1}{2})\tilde{t}_r^{\alpha} 
% \langle\langle\tau_{r+n-j}(c_1(X)^{j}\cdot\phi_\alpha);\mathsf{P}_g^g(0)\rangle\rangle_{g}
% \\&+\frac{1}{2}\sum_{j=0}^{n+1}\sum_{s,\alpha}(-1)^{-s-1}e_{n+1-j}(-s+b^{\alpha}-\frac{3}{2},...,-s+n+b^{\alpha}-\frac{3}{2})
% \\&\quad\quad\cdot\sum_{g_1+g_2=g}\langle\langle\tau_{-s-1+n-j}(c_1(X)^j\cdot\phi^\alpha);\mathsf{P}_{g_1}^{g_1}(0)\rangle\rangle_{g_1}
% \langle\langle\tau_{s}(\phi_\alpha);\mathsf{P}_{g_2}^{g_2}(0)\rangle\rangle_{g_2}  
% \\&+\frac{1}{2}\delta_{g}^{0}\sum_{\alpha,\beta}(\mathcal{C}^{n+1})_{\alpha\beta}t_0^\alpha t_0^\beta.
% \end{align*}
\begin{theorem}\label{thm:new-type-uni-constr-descen-gw}
For any smooth projective varieties with semisimple quantum cohomolgy or smooth algebraic curves, 
for any $m$ and 
$\beta$, the genus-$g$ pure descendant Gromov-Witten invariants of $X$  satisfy the following
 identity
\begin{align*}
\Theta^{\mathsf{P}}_{g,n,m,\beta}=0.
\end{align*}
    
\end{theorem}

As an application of the explicit formula~\eqref{eqn:theta-Pixton-class-gnmbeta} of $\Theta^{\mathsf{P}}_{g,n,m,\beta}$, we prove
\begin{theorem}\label{thm:lambda-g-conj-g=1}
For any smooth projective variety, the $\lambda_{g}$ conjecture holds for genus one Hodge integrals of $X$  with $\lambda_{g}$ insertion.  
\end{theorem}

This paper is organized as follows. In section~\ref{sec:virasoro-conj-Hodge}, we recall the basic knowledge in Gromov-Witten theory, Hodge integrals, and $\lambda_{g}$ conjecture. In section~\ref{sec:Virasoro-gw-Hodge}, 
% we give the Virasoro operators for Gromov-Witten invariants and Hodge integrals and prove Theorem~\ref{thm:genus-gdegree-gL-n-equi}. In section~\ref{sec:fiber-degree-0-Virasoro conjecture},
we study the fiber degree-0 Virasoro conjecture for $X\times\mathbb{P}^1$ and prove Theorem~\ref{thm:virasoro-lambda-g-conj-rela} and \ref{thm:lambda-g-conj-g=0}. 
In section~\ref{sec:Givental type formula}, we prove the graph sum formula of semisimple CohFT times $\lambda_{g}$ class and the explicit graph sum formula of total ancestor potential of $\lambda_{g}$ invariants.  In section~\ref{sec:new-universal-gw-pixconj}, we review Pixton's formula for double ramification cycles and give the proof of Theorem~\ref{thm:new-type-uni-constr-descen-gw} and \ref{thm:lambda-g-conj-g=1}.   

\section{Gromov-Witten
invariants and Hodge integrals}\label{sec:virasoro-conj-Hodge}
\subsection{Gromov-Witten
invariants}
Let $X$ be a smooth projective variety of complex dimension $d$, $\{\phi_1,...,\phi_N\}$ be a graded basis of
$H^*(X;\mathbb{C})$, such that $\phi_1$ is the identity element and $\phi_\alpha\in H^{p_\alpha,q_\alpha}(X;\mathbb{C})$. Let $\eta=(\eta_{\alpha\beta})$, $\eta_{\alpha\beta}:=\int_{X}\phi_\alpha\cup\phi_\beta$ be the intersection pairing on $H^*(X;\mathbb{C})$  and $\eta^{-1}=(\eta^{\alpha\beta})$ be the inverse of matrix $\eta$.
We will use $\eta$ and $\eta^{-1}$ to lower and raise indices.

Recall that the big phase space is by definition the product of infinite copies of $H^*(X;\mathbb{C})$, that is $\mathcal{P}:=\prod_{n=0}^{\infty}H^*(X;\mathbb{C})$.  Then we denote the corresponding basis for the $n$-th copy of $H^*(X;\mathbb{C})$
 in $\mathcal{P}$ by $\{\tau_{n}(\phi_\alpha)\}$.   We can think of $\mathcal{P}$ as an infinite dimensional vector space with basis $\{\tau_n(\phi_\alpha):1\leq \alpha\leq N,n\geq0\}$. Let $\{t_n^\alpha:n\geq0,\alpha=1,...,N\}$ be
the corresponding coordinate system on $\mathcal{P}$. For convenience, we identify $\tau_n(\phi_\alpha)$ with the
coordinate vector field $\frac{\partial}{\partial t_n^\alpha}$
on $\mathcal{P}$ for $n\geq0$. If $n<0$, $\tau_n(\phi_\alpha)$ is understood as the 0 vector
field. We also abbreviate $\tau_n(\phi_\alpha)$ as $\phi_\alpha$.  We use $\tau_+$ and $\tau_-$
to denote the operator which shift the level of descendants, i.e.
\[\tau_{\pm}\left(\sum_{n,\alpha}f_{n,\alpha}\tau_{n}(\phi_\alpha)\right):=\sum_{n,\alpha}f_{n,\alpha}\tau_{n\pm 1}(\phi_\alpha)\]
where $f_{n,\alpha}$ are functions on the big phase space.
Instead of coordinates $\{t_n^\alpha\}$, it is very convenient to use the
following shifted coordinates on the big phase space $$\Tilde{t}_n^\alpha=t_n^\alpha-\delta_n^1\delta_\alpha^1.$$

Let $\overline{\mathcal{M}}_{g,n}(X,B)$ be the moduli space of stable maps $f:(C;x_1,...,x_n)\rightarrow X$,  where $(C;x_1,...,x_n)$
is a genus-$g$ nodal curve with $n$ marked points and $f_*([C])=B\in H_2(X;\mathbb{Z})$.
The descendant Gromov-Witten invariants is defined to be
\begin{align}\label{eqn:def-des-gw-inv}
\langle\tau_{k_1}(\phi_{\alpha_1})...\tau_{k_n}(\phi_{\alpha_n})\rangle_{g,n}:=\sum_{B\in H_2(X;\mathbb{Z})}q^{B}\int_{[\overline{\mathcal{M}}_{g,n}(X,B)]_{vir}} \prod_{i=1}^{n}ev_i^*\phi_{\alpha_i}\cup c_1(L_i)^{k_i}   
\end{align}
where $q$ is the Novikov variable, $L_i$ are the tautological line bundles over $\overline{\mathcal{M}}_{g,n}(X,B)$, and $ev_i: \overline{\mathcal{M}}_{g,n}(X,B)\rightarrow X$ are
the evaluation maps. Let  $\mathbf{t}$ be the sets of variables $\{t_n^\alpha,n\geq0, \alpha=1,...,N\}$. 
We
define the generating
function of genus-$g$ Gromov-Witten invariants
\begin{align*}
% \label{eqn:def0Fggw}
F_g(\mathbf{t}):=\langle\exp(\sum_{n,\alpha}t_n^\alpha\tau_n(\phi_\alpha))\rangle_{g} 
\end{align*}
and its derivatives
\begin{align*}
\langle\langle\tau_{k_1}(\phi_{\alpha_1})...\tau_{k_n}(\phi_{\alpha_n})\rangle\rangle_{g}
:=\frac{\partial^n}{\partial t_{k_1}^{\alpha_1}...\partial t_{k_n}^{\alpha_n}}F_{g}(\mathbf{t}).
\end{align*}
For our purpose, we also define covariant derivatives of $F_g$ with respect to the trivial
connection on $\mathcal{P}$ by
\[\langle\langle \mathcal{W}_1,...,\mathcal{W}_n\rangle\rangle_g:=\sum_{k_1,\alpha_1,...,k_n,\alpha_n}f^{1}_{k_1,\alpha}...f^n_{k_n,\alpha_n}\frac{\partial^n}{\partial t_{k_1}^{\alpha_1}...\partial t_{k_n}^{\alpha_n}}F_{g}(\mathbf{t})\]
for vector fields $\mathcal{W}_i=\sum_{k,\alpha}f_{k,\alpha}^{i}\tau_{k}(\phi_\alpha)$ on the big phase space. Using this double bracket, we define the quantum product on the big phase space: for any vector fields $\mathcal{W}_1$ and $\mathcal{W}_2$,
\begin{align*}
% \label{eqn:quantum-product-bps}
\mathcal{W}_1\bullet\mathcal{W}_2=\sum_{\alpha}\langle\langle\mathcal{W}_1\mathcal{W}_2\phi_\alpha\rangle\rangle_0\phi^\alpha.
\end{align*}
The total descendant potential is defined as:
\[\mathcal{D}(\mathbf{t}):=\exp(\sum_{g\geq0}\hbar^{2g-2}F_g(\mathbf{t}))
\]
where $\hbar$ is a formal parameter.
\subsection{Hodge integrals}
The
Hodge bundle $\mathbb{E}$ over $\overline{\mathcal{M}}_{g,n}(X,B)$ is the rank $g$ vector bundle with fiber $H^0(C,\omega_C)$ over the domain curve $(C;x_1,...,x_n)$, where $\omega_C$ is the dualizing sheaf of curve $C$. Let $\lambda_{k}$ be
the $k$-th Chern class of $\mathbb{E}$.
The Hodge integral is defined to be 
\begin{align}\label{eqn-def:Hodge-cor}
\langle\tau_{k_1}(\phi_{\alpha_1})...\tau_{k_n}(\phi_{\alpha_n});\prod_{j=1}^{m}\lambda_{l_j}\rangle_{g,n}:=\sum_{B\in H_2(X;\mathbb{Z})}q^{B}\int_{[\overline{\mathcal{M}}_{g,n}(X,B)]_{vir}} \prod_{i=1}^{n}ev_i^*\phi_{\alpha_i}\cup c_1(L_i)^{k_i}   \cup \prod_{j=1}^{m}\lambda_{l_j}.   
\end{align}
By definition, the descendant Gromov-Witten invariants~\eqref{eqn:def-des-gw-inv} can be recovered from Hodge integrals~\eqref{eqn-def:Hodge-cor} by taking $m=0$.

The generating
function of genus-$g$ Hodge integrals is defined to be
\begin{align}\label{eqn:Fg-Hodge-t-s}
F_g^{\mathbb{E}}(\mathbf{t},\mathbf{s})
:=
\Bigg\langle\exp\left(\sum_{n,\alpha}t_n^\alpha\tau_n(\phi_\alpha)\right);\exp\left(\sum_{m\geq0}s_{m}\lambda_{m}\right)\Bigg\rangle_g.
\end{align}
In particular, after the restriction $\mathbf{s}=0$, the generating
function $F_g(\mathbf{t}):=F_{g}^{\mathbb{E}}|_{\mathbf{s}=0}$ becomes
the usual generating function of genus-$g$ pure descendant
Gromov--Witten invariants.
For convenience, we identify  $\lambda_n$ with $\frac{\partial}{\partial s_n}$ and use the following double brackets denote differentiation of $F_{g}^{\mathbb{E}}(\mathbf{t},\mathbf{s})$,
\begin{align*}
\langle\langle\tau_{k_1}(\phi_{\alpha_1})\ldots\tau_{k_n}(\phi_{\alpha_n});\prod_{i=1}^{m}\lambda_{l_i}\rangle\rangle_{g}^{\mathbb{E}} 
=\frac{\partial^{n+m}}{\partial t_{k_1}^{\alpha_1}\dots\partial t_{k_n}^{\alpha_n}\partial s_{l_1}\dots\partial s_{l_m}}F_{g}^{\mathbb{E}}(\mathbf{t},\mathbf{s})
\end{align*}
and
\begin{align*}
\langle\langle\tau_{k_1}(\phi_{\alpha_1})\ldots\tau_{k_n}(\phi_{\alpha_n});\prod_{i=1}^{m}\lambda_{l_i}\rangle\rangle_{g}
=\langle\langle\tau_{k_1}(\phi_{\alpha_1})\ldots\tau_{k_n}(\phi_{\alpha_n});\prod_{i=1}^{m}\lambda_{l_i}\rangle\rangle_{g}^{\mathbb{E}}|_{\mathbf{s}=0}.\end{align*}
\subsection{Virasoro constraints for descendant Gromov-Witten invariants}\label{subsec:virasoro-gw-inva}
In this subsection, we recall the constructions of Virasoro operators by Eguchi, Hori, and
Xiong, modified by Katz (cf. \cite{eguchi1997quantum}, \cite{cox1999mirror}).
For our purpose, we define 
\[b_\alpha:=p_\alpha-\frac{d-1}{2},\,\, b^\alpha:=1-b_\alpha\]
and functions
\[\hat{e}_j(n,\alpha)(x)=e_j(x+b_\alpha,...,x+n+b_\alpha),\quad \check{e}_j(n,\alpha)(x)=e_j(x+b^\alpha,...,x+n+b^\alpha)\]
where $e_j$ is the $j$-th elementary symmetric function. Let $\mathcal{C}=(\mathcal{C}_\alpha^\beta)$ be the matrix of
multiplication by the first Chern class $c_1(X)$ in the ordinary cohomology ring, i.e.
$c_1(X)\cup \phi_\alpha=\sum_{\beta}\mathcal{C}_\alpha^\beta\phi_\beta$.  
The differential operators are defined to be
\begin{align*}
L_n
=&\sum_{j=0}^{n+1}\sum_{r,\alpha}\left(\hat{e}_{n+1-j}(n,\alpha)\right)(r) \tilde{t}_r^\alpha \tau_{r+n-j}(c_1(X)^{j}\cup\phi_\alpha)
\\&+\frac{\hbar^2}{2}\sum_{j=0}^{n+1}\sum_{s,\alpha}(-1)^{s+1}\left(\check{e}_{n+1-j}(n,\alpha)\right)(-s-1)\tau_s(\phi_\alpha)\tau_{-s-1+n-j}(c_1(X)^{j}\cup\phi^{\alpha})
\\&+\frac{1}{2\hbar^2}\sum_{\alpha,\beta}(\mathcal{C}^{n+1})_{\alpha\beta}t_0^\alpha t_0^\beta
+\frac{\delta_{n,0}}{24}
\int_{X}\left(\frac{3-d}{2}c_d(X)-c_1(X)\cup c_{d-1}(X)\right),\quad n\geq-1
\end{align*}
which satisfy Virasoro bracket  relation
\begin{align*}
[L_n,L_m]=(n-m)L_{m+n},\quad n\geq-1.    
\end{align*}
Here $\mathcal{C}^j$
is the $j$-th power of the matrix $\mathcal{C}$, $(\mathcal{C}^j)^{\alpha\beta}$ are entries of the matrix $(\mathcal{C}^{j})\eta^{-1}$ and $(\mathcal{C}^{n+1})_{\alpha\beta}$ are entries of the matrix $(\mathcal{C}^{n+1})\eta$.

Virasoro conjecture states that: 
For any smooth projective variety, 
\[L_n\mathcal{D}(\mathbf{t})=0, \quad n\geq-1.\]
It is well known that for any compact symplectic manifold $L_n\mathcal{D}(\mathbf{t})=0$ for $n=-1$ or $0$.  The first
equation (i.e. for $n=-1$) is the string equation. The second equation (i.e. for $n=0$)
is derived from  divisor equation, dilaton equation and selection rule.
\subsection{$\lambda_{g}$ conjecture}
The 
$\lambda_{g}$
conjecture gives a particularly simple formula for certain integrals on  the moduli space of stable curves with marked points 
$\overline{\mathcal{M}}_{g,n}$. It was first found as a consequence of the Virasoro conjecture for Gromov-Witten invariants of $\mathbb{P}^1$ by  Getzler and  Pandharipande in \cite{getzler1998virasoro}.  Later, it was proven by Faber and  Pandharipande (cf. \cite{faber2003hodge}) using virtual localization in Gromov–Witten theory. It is named after the factor of 
$\lambda_{g}$, the $g$-th Chern class of the Hodge bundle, appearing in its integrand. The other factor is a monomial in the $\psi_i$
, the first Chern classes of the $n$ cotangent line bundles, as in Witten's conjecture.

To be precise, let $a_1,...,a_n$
 be positive integers such that $\sum_{i=1}^{n}a_i=2g-3+n$, then the 
$\lambda_g$ conjecture/theorem  can be stated as follows:
\begin{align}\label{eqn:lambda-g-conjecture-pt}
\int_{\overline{\mathcal{M}}_{g,n}}\psi_1^{a_1}...\psi_{n}^{a_n}\lambda_{g}
=\binom{2g-3+n}{a_1,...,a_n}b_g  \end{align}
where
\begin{align*}
b_g=\begin{cases}
1,\quad g=0\\ 
\int_{\overline{\mathcal{M}}_{g,1}}\psi_1^{2g-2}\lambda_{g},\quad g\geq1  
\end{cases}    
\end{align*}

Take the target variety $X=\text{point}$, then $H^*(X;\mathbb{C})=\mathbb{C}$, with $N=1$ and $\phi_1=\mathbf{1}$ is the identity.
For convenience, we denote  coordinates $t_n^1$ by $t_n$ and the correlation function $\langle\langle\tau_{k_1}(\phi_{1})\ldots\tau_{k_n}(\phi_{1});\prod_{i=1}^{m}\lambda_{l_i}\rangle\rangle_{g}$ by $\langle\langle\tau_{k_1}\ldots\tau_{k_n};\prod_{i=1}^{m}\lambda_{l_i}\rangle\rangle_{g}$. 
It was proved in
\cite{getzler1998virasoro} that $\lambda_{g}$ formula~\eqref{eqn:lambda-g-conjecture-pt} is equivalent to  a sequence of equations:
\begin{align*}
% \label{eqn:seq-lambda-g-equi}
\sum_{r\geq1}\frac{\Gamma(r+n+1)}{\Gamma(r)} (t_r-\delta_{r}^{1})\langle\langle\tau_{r+n}\tau_{m};\lambda_{g}\rangle\rangle_{g}+\frac{\Gamma(m+n+2)}{\Gamma(m+1)}\langle\langle\tau_{m+n};\lambda_{g}\rangle\rangle_{g}=0,\quad n\geq1
\end{align*}
where $\Gamma(x)$ is Gamma function. Later, Jiang and Tseng proved a generalized $\lambda_{g}$ theorem of orbifold version, which involves descendant cyclic Hurwitz
Hodge integrals over stacky point (cf. \cite{jiang2010virasoro}).
In this paper, we would like to generalize this sequence of constraints to arbitrary smooth projective varieties, which we call $\lambda_g$ constraints/conjecture.

\section{Virasoro constraints for  Gromov-Witten invariants of product space $X\times\mathbb{P}^1$}\label{sec:Virasoro-gw-Hodge}
Denote the hyperplane class of $\mathbb{P}^1$ by $H\in H^2(\mathbb{P}^1;\mathbb{C})$.
Since the cohomology ring  $H^*(Y;\mathbb{C})\cong H^*(X;\mathbb{C})\otimes_{\mathbb{C}} H^*(\mathbb{P}^1;\mathbb{C})$, there is a canonical basis $\{\tau_n(\phi_\alpha\otimes1), \tau_n(\phi_\alpha\otimes H):1\leq\alpha\leq N, n\geq0\}$ of the big phase space $\prod_{n=0}^{\infty}H^*(Y;\mathbb{C})$. Denote the flat coordinates of the big phase space $\prod_{n=0}^{\infty}H^*(Y;\mathbb{C})$ by $\{t_n^\alpha, u_n^\alpha: 1\leq \alpha\leq N,n\geq0\}$ respect to the canonical basis.  
For convenience, we introduce new notations: $\phi_{(\alpha,0)}=\phi_\alpha\otimes1$, $\phi_{(\alpha,1)}=\phi_\alpha\otimes H$ and its dual basis  
$\phi^{(\alpha,0)}=\phi^\alpha\otimes H$, $\phi^{(\alpha,1)}=\phi^\alpha\otimes1$. We also introduce new notations for the flat coordinates: $t_n^{(\alpha,0)}=t_n^\alpha$ and $t_{n}^{(\alpha,1)}=u_n^\alpha$. 

Define a product space $Y=X\times\mathbb{P}^1$, which can be seen as a trivial $\mathbb{P}^1$ bundle over $X$. We first write down the explicit Virasoro operators for Gromov-Witten invariants of $Y$. Then we establish the relations between Gromov-Witten invariants of $Y$ and Hodge integrals of $X$.

\subsection{Virasoro operators for Gromov-Witten invariants of $X\times\mathbb{P}^1$}
\begin{lemma}
\label{lem:c1-Y-cd-Y}
In cohomology ring $H^*(Y;\mathbb{C})=H^*(X;\mathbb{C})\otimes_{\mathbb{C}} H^*(\mathbb{P}^1;\mathbb{C})$
\begin{align*}
(i)& \quad c_1(Y)^j=c_1(X)^j\otimes1+ j\cdot c_1(X)^{j-1}\otimes(2H), \quad j\geq0 \\
(ii)&\quad c_{d+1}(Y)=c_{d}(X)\otimes(2H)\\
(iii)&\quad c_d(Y)
=c_{d}(X)\otimes1+c_{d-1}(X)\otimes (2H). 
\end{align*}   
\end{lemma}
\begin{proof}
Note that
\begin{align*}
c_1(Y)=c_1(X)\otimes1+1\otimes(2H),    
\end{align*}
then 
\begin{align*}
c_1(Y)^j=c_1(X)^j\otimes1+ j\cdot c_1(X)^{j-1}\otimes(2H). \end{align*}
Denote $c(X)$ and $c(\mathbb{P}^1)$ to be the total Chern class of $X$ and $\mathbb{P}^1$, then we have
\begin{align*}
c_{d+1}(Y)=[c(X)\otimes c(\mathbb{P}^1)]_{\deg=d+1}=c_{d}(X)\otimes(2H) 
\end{align*}
and
\begin{align*}
c_d(Y)
=[c(X)\otimes c(\mathbb{P}^1)]_{\deg=d}
=c_{d}(X)\otimes1+c_{d-1}(X)\otimes (2H).
\end{align*} 
This lemma is proved.
\end{proof}

\begin{lemma}
\begin{align*}
(i)&\quad b_{(\alpha,0)}^Y
=b_\alpha^X-\frac{1}{2},\quad   b_{(\alpha,1)}^Y
=b_\alpha^X+\frac{1}{2}\\
(ii)&\quad b_Y^{(\alpha,0)}=b^\alpha_X+\frac{1}{2},\quad  b_Y^{(\alpha,1)}=b_X^\alpha-\frac{1}{2}
\end{align*}    
\end{lemma}

\begin{proof}
By definition,
\begin{align*}
b_{(\alpha,0)}^Y=p_{(\alpha,0)}-\frac{\dim(Y)-1}{2}  =p_\alpha-\frac{\dim(X)}{2}
=b_\alpha^X-\frac{1}{2},
\end{align*}
and
\begin{align*}
b_{(\alpha,1)}^Y=p_{(\alpha,1)}-\frac{\dim(Y)-1}{2}  =p_\alpha+1-\frac{\dim(X)}{2}
=b_\alpha^X+\frac{1}{2}.
\end{align*}    

% \begin{align*}
% b_{(\alpha,0)}=b_\alpha-\frac{1}{2},\, b_{(\alpha,1)}=b_\alpha+\frac{1}{2},\, b^{(\alpha,1)}=b^\alpha-\frac{1}{2}    
% \end{align*}
% and
% here 
% \begin{align*}
% &b^{(\alpha,0)} =1-b_{(\alpha,0)}
% =\frac{3}{2}-b_\alpha^{X}
% \end{align*}
\end{proof}

\begin{proposition}\label{prop:Ln-expression}For $n\geq-1$,
the Virasoro operators for Gromov-Witten invariants of $Y$ has the expression
\begin{align}\label{eqn:prop-Ln-expression}
&L_n
\nonumber\\=&\sum_{j=0}^{n+1}\sum_{j,\alpha}e_{n+1-j}(r+b_{\alpha}-\frac{1}{2},...,r+n+b_{\alpha}-\frac{1}{2})\tilde{t}_r^{\alpha} \tau_{r+n-j}((c_1(X)^{j}\cup\phi_{\alpha})\otimes1)
\nonumber\\&+\sum_{j=0}^{n+1}\sum_{r,\alpha}(2j)\cdot e_{n+1-j}(r+b_\alpha-\frac{1}{2},...,r+n+b_\alpha-\frac{1}{2})\tilde{t}_r^\alpha\tau_{r+n-j}((c_1(X)^{j-1}\cup\phi_\alpha)\otimes H)
\nonumber\\&+\sum_{j=0}^{n+1}\sum_{r,\alpha}e_{n+1-j}(r+b_\alpha+\frac{1}{2},...,r+n+b_\alpha+\frac{1}{2})u_{r}^{\alpha}\tau_{r+n-j}((c_1(X)^j\cup\phi_\alpha)\otimes H)
\nonumber\\&+\hbar^2\sum_{j=0}^{n+1}\sum_{s,\alpha}(-1)^{-s-1}e_{n+1-j}(-s-\frac{3}{2}+b^\alpha,...,-s-\frac{3}{2}+n+b^\alpha)
\nonumber\\&\hspace{60pt}\cdot\tau_s(\phi_{\alpha}\otimes H)\tau_{-s-1+n-j}((c_1(X)^{j}\cup\phi^{\alpha})\otimes1)
\nonumber\\&+\frac{\hbar^2}{2}\sum_{j=0}^{n+1}\sum_{s,\alpha}(-1)^{-s-1}(2j)\cdot e_{n+1-j}(-s-\frac{3}{2}+b^{\alpha},...,-s-\frac{3}{2}+n+b^{\alpha})
\nonumber\\&\hspace{60pt}\cdot\tau_s(\phi_{\alpha}\otimes H)\tau_{-s-1+n-j}((c_1(X)^{j-1}\cup\phi^{\alpha})\otimes H) 
\nonumber\\&+\frac{n+1}{\hbar^2}\sum_{\alpha,\beta}t_0^{\alpha} t_0^{\beta}\int_{X}c_1(X)^{n}\cup\phi_\alpha\cup\phi_\beta
+\frac{1}{\hbar^2}\sum_{\alpha,\beta}t_0^{\alpha} u_0^{\beta}\int_{X}c_1(X)^{n+1}\cup\phi_\alpha\cup\phi_\beta
\nonumber\\&+\frac{\delta_{n,0}}{24}\int_{X}\left((-d)c_d(X)-2c_1(X)\cup c_{d-1}(X)\right)
\end{align}    where $c_1(X)$ is the first Chern class of $X$, $b_\alpha=b_\alpha^{X}$ and $b^\alpha=b^\alpha_{X}$. 
% and $\mathcal{C}=(\mathcal{C}_\alpha^\beta)$ be the matrix of
% multiplication by the first Chern class $c_1(X)$ in the ordinary cohomology ring $H^*(X;\mathbb{C})$ of $X$.
\end{proposition}
\begin{proof}
By definition of 
Virasoro operators in section~\ref{subsec:virasoro-gw-inva}, we have
\begin{align*}
&L_n
\\=&\sum_{j=0}^{n+1}\sum_{r,\alpha}e_{n+1-j}(r+b_{(\alpha,0)},...,r+n+b_{(\alpha,0)})\tilde{t}_r^{\alpha} \tau_{r+n-j}(c_1(Y)^{j}\cup\phi_{(\alpha,0)})
\\&+\sum_{j=0}^{n+1}\sum_{r,\alpha}e_{n+1-j}(r+b_{(\alpha,1)},...,r+n+b_{(\alpha,1)})u_r^{\alpha} \tau_{r+n-j}(c_1(Y)^{j}\cup\phi_{(\alpha,1)})
\\&+\frac{\hbar^2}{2}\sum_{j=0}^{n+1}\sum_{s,\alpha}(-1)^{-s-1}e_{n+1-j}(-s-1+b^{(\alpha,0)},...,-s-1+n+b^{(\alpha,0)})
\\&\hspace{80pt}\cdot\tau_s(\phi_{\alpha,0})\tau_{-s-1+n-j}(c_1(Y)^{j}\cup\phi^{(\alpha,0)})
\\&+\frac{\hbar^2}{2}\sum_{j=0}^{n+1}\sum_{s,\alpha}(-1)^{-s-1}e_{n+1-j}(-s-1+b^{(\alpha,1)},...,-s-1+n+b^{(\alpha,1)})
\\&\hspace{80pt}\cdot\tau_s(\phi_{\alpha,1})\tau_{-s-1+n-j}(c_1(Y)^{j}\cup\phi^{(\alpha,1)})
\\&+\frac{1}{2\hbar^2}\sum_{\alpha,\beta,i,j}\left(\int_{Y}c_1(Y)^{n+1}\cup\phi_{(\alpha,i)}\cup\phi_{(\beta,j)}\right)t_0^{(\alpha,i)} t_0^{(\beta,j)}
\\&+\frac{\delta_{n,0}}{24}\int_{Y}\left(\frac{3-(d+1)}{2}c_{d+1}(Y)-c_1(Y)\cup c_{d}(Y)\right).
\end{align*}
Then we compute each terms explicitly by Lemma~\ref{lem:c1-Y-cd-Y}. In the first term 
\begin{align*}
&c_1(Y)^j\cup\phi_{(\alpha,0)} = (c_1(X)^j\otimes1+ j\cdot c_1(X)^{j-1}\otimes(2H))\cup(\phi_\alpha\otimes1)
\nonumber\\=&(c_1(X)^j\cup\phi_\alpha)\otimes1+ (2j)\cdot (c_1(X)^{j-1}\cup\phi_\alpha)\otimes H.
\end{align*}
In the second term,
\begin{align*}
&c_1(Y)^j\cup\phi_{(\alpha,1)} = (c_1(X)^j\otimes1+ j\cdot c_1(X)^{j-1}\otimes(2H))\cup(\phi_\alpha\otimes H)
\nonumber\\=&(c_1(X)^j\cup\phi_\alpha)\otimes H.
\end{align*}
For the third term, 
\begin{align*}
&c_1(Y)^j\cup\phi^{(\alpha,0)}
=(c_1(X)^j\otimes1+ j\cdot c_1(X)^{j-1}\otimes(2H))\cup(\phi^\alpha\otimes H)
\\=&(c_1(X)^j\cup\phi^\alpha)\otimes H.
\end{align*}
For the fourth term,
\begin{align*}
&c_1(Y)^j\cup\phi^{(\alpha,1)}
=(c_1(X)^j\otimes1+ j\cdot c_1(X)^{j-1}\otimes(2H))\cup(\phi^\alpha\otimes 1)
\\=&(c_1(X)^j\cup\phi^\alpha)\otimes1+(2j) (c_1(X)^{j-1}\cup\phi^\alpha)\otimes H.
\end{align*}
For the fifth term,  
\begin{align*}
&\sum_{\alpha,\beta,i,j}\left(\int_{Y}c_1(Y)^{n+1}\cup\phi_{(\alpha,i)}\cup\phi_{(\beta,j)}\right)t_0^{(\alpha,i)} t_0^{(\beta,j)}
\\=&\sum_{\alpha,\beta}\int_{X\times\mathbb{P}^1}c_1(Y)^{n+1}\cup\phi_{(\alpha,0)}\cup\phi_{(\beta,0)}t_0^{(\alpha,0)}t_0^{(\beta,0)}
+2\sum_{\alpha,\beta}\int_{X\times\mathbb{P}^1}c_1(Y)^{n+1}\cup\phi_{(\alpha,0)}\cup\phi_{(\beta,1)}t_0^{(\alpha,0)}t_0^{(\beta,1)}
\\&+\sum_{\alpha,\beta}\int_{X\times\mathbb{P}^1}c_1(Y)^{n+1}\cup\phi_{(\alpha,1)}\cup\phi_{(\beta,1)}t_0^{(\alpha,1)}t_0^{(\beta,1)}
\\=&2(n+1) \sum_{\alpha,\beta} \left(\int_{X}c_1(X)^{n}\cup\phi_{\alpha}\cup\phi_{\beta}\right)t_0^{\alpha}t_0^{\beta}
+2\sum_{\alpha,\beta}\left(\int_{X}c_1(X)^{n+1}\cup\phi_{\alpha}\cup\phi_{\beta}\right)t_0^{\alpha}u_0^{\beta}.
\end{align*}
For the last constant term, by Lemma~\ref{lem:c1-Y-cd-Y}, we have
\begin{align*}
\int_{Y}c_1(Y)\cup c_{d}(Y) =2\int_{X}c_1(X)\cup c_{d-1}(X)+2\int_{X}c_d(X),   
\end{align*}
and
\begin{align*}
\int_{Y}c_{d+1}(Y)
=2\int_{X}c_d(X).  
\end{align*}
So we get
\begin{align*}
\int_{Y}\left(\frac{3-(d+1)}{2}c_{d+1}(Y)-c_1(Y)\cup c_{d}(Y)\right)
=\int_{X}\left((-d)c_d(X)-2c_1(X)\cup c_{d-1}(X)\right).
\end{align*}
Therefore, after simplification, \begin{align*}
&L_n
\\=&\sum_{j=0}^{n+1}\sum_{j,\alpha}e_{n+1-j}(r+b_{\alpha}-\frac{1}{2},...,r+n+b_{\alpha}-\frac{1}{2})\tilde{t}_r^{\alpha} \tau_{r+n-j}((c_1(X)^{j}\cup\phi_{\alpha})\otimes1)
\\&+\sum_{j=0}^{n+1}\sum_{r,\alpha}(2j)\cdot e_{n+1-j}(r+b_\alpha-\frac{1}{2},...,r+n+b_\alpha-\frac{1}{2})\tilde{t}_r^\alpha\tau_{r+n-j}((c_1(X)^{j-1}\cup\phi_\alpha)\otimes H)
\\&+\sum_{j=0}^{n+1}\sum_{r,\alpha}e_{n+1-j}(r+b_\alpha+\frac{1}{2},...,r+n+b_\alpha+\frac{1}{2})u_{r}^{\alpha}\tau_{r+n-j}((c_1(X)^j\cup\phi_\alpha)\otimes H)
\\&+\frac{\hbar^2}{2}\sum_{j=0}^{n+1}\sum_{s,\alpha}(-1)^{-s-1}e_{n+1-j}(-s-\frac{1}{2}+b^\alpha,...,-s-\frac{1}{2}+n+b^\alpha)
\\&\hspace{80pt}\cdot\tau_s(\phi_{\alpha}\otimes1)\tau_{-s-1+n-j}((c_1(X)^{j}\cup\phi^{\alpha})\otimes H)
\\&+\frac{\hbar^2}{2}\sum_{j=0}^{n+1}\sum_{s,\alpha}(-1)^{-s-1}e_{n+1-j}(-s-\frac{3}{2}+b^\alpha,...,-s-\frac{3}{2}+n+b^\alpha)\\&\hspace{60pt}\cdot\tau_s(\phi_{\alpha}\otimes H)\tau_{-s-1+n-j}((c_1(X)^{j}\cup\phi^{\alpha})\otimes1)
\\&+\frac{\hbar^2}{2}\sum_{j=0}^{n+1}\sum_{s,\alpha}(-1)^{-s-1}(2j)\cdot e_{n+1-j}(-s-\frac{3}{2}+b^{\alpha},...,-s-\frac{3}{2}+n+b^{\alpha})\\&\hspace{60pt}\cdot\tau_s(\phi_{\alpha}\otimes H)\tau_{-s-1+n-j}((c_1(X)^{j-1}\cup\phi^{\alpha})\otimes H) 
\\&+\frac{n+1}{\hbar^2}\sum_{\alpha,\beta}t_0^{\alpha} t_0^{\beta}\int_{X}c_1(X)^{n}\cup\phi_\alpha\cup\phi_\beta
+\frac{1}{\hbar^2}\sum_{\alpha,\beta}t_0^{\alpha}u_0^{\beta}
\int_{X}c_1(X)^{n+1}\cup\phi_\alpha\cup\phi_\beta \\&+\frac{\delta_{n,0}}{24}\int_{X}\left((-d)c_d(X)-2c_1(X)\cup c_{d-1}(X)\right)
\end{align*}    where $b_\alpha=b_\alpha^{X}$ and $b^\alpha=b_X^\alpha$. In fact, the fourth term and fifth term are equal to each other. To see this, taking transformation of index $s=-s'-1+n-j$ and using the basic fact that $(\mathcal{C}^j)^{\alpha\beta}\neq0$ implies $j+b^\alpha+b^\beta=1$,  it is direct to compute the fourth term
\begin{align*}
&\sum_{s,\alpha}(-1)^{-s-1}e_{n+1-j}(-s-\frac{1}{2}+b^\alpha,...,-s-\frac{1}{2}+n+b^\alpha)
\\&\hspace{40pt}\cdot\tau_{s}(\phi_\alpha\otimes1)\tau_{-s-1+n-j}((c_1(X)^j\cup\phi^\alpha)\otimes H)
\\=&\sum_{s',\alpha,\beta}(-1)^{s'-n+j}e_{n+1-j}(s'-n+\frac{3}{2}-b^\beta,...,s'+\frac{3}{2}-b^\beta)
\\&\hspace{40pt}\cdot\tau_{-s'-1+n-j}(\phi_\alpha\otimes1)\tau_{s'}(\phi_\beta\otimes H)(\mathcal{C}^{j})^{\alpha\beta}
\\=&\sum_{s',\beta}(-1)^{s'+1}e_{n+1-j}(-s'+n-\frac{3}{2}+b^\beta,...,-s'-\frac{3}{2}+b^\beta)
\\&\hspace{40pt}\cdot\tau_{-s'-1+n-j}((c_1(X)^j\cup\phi^\beta)\otimes1)\tau_{s'}(\phi_\beta\otimes H).
\end{align*}
As second order derivative operators, this match with the fifth term. After combining the fourth term and fifth term, we get equation~\eqref{eqn:prop-Ln-expression}.
   
\end{proof}
\subsection{The fiber degree 0 Gromov-Witten invariants of $Y$}
The following proposition plays a crucial role in studying the relations between Hodge integrals of $X$ and Gromov-Witten invariants of $Y$. 
\begin{proposition}\label{prop:deg-0-vir-fun-class}
Let $X=X_1\times X_2$ be a product of smooth projective varieties. Let
 $(B,0)\in H_2(X;\mathbb{Z})$ be the image of a class 
$B\in H_2(X_1;\mathbb{Z})$ under the fiber map $X_1\rightarrow X$. Then
$\overline{\mathcal{M}}_{g,n}(X_1\times X_2, (B,0))=\overline{\mathcal{M}}_{g,n}(X_1, B)\times X_2$ and the corresponding  virtual fundamental classes are related
 by
 \begin{align*}
&[\overline{\mathcal{M}}_{g,n}(X_1\times X_2,(B,0))]_{vir}=  \left([\overline{\mathcal{M}}_{g,n}(X_1, B)]_{vir}\times[X_2]\right)\cap e(\mathbb{E}^{*}\boxtimes TX_2).
\end{align*}
\end{proposition}
\begin{proof}
Let $\mathfrak{M}_{g,n}$ be the algebraic
 stack of $n$-marked prestable curves of genus $g$. This is an algebraic stack,
 not of Deligne-Mumford (or even finite) type, but smooth of dimension
 $3g-3+n$. There is a canonical morphism
 $\overline{\mathcal{M}}_{g,n}(X_1\times X_2,(B,0))\rightarrow\mathfrak{M}_{g,n}$,
 given by forgetting the map, retaining the curve (but not stabilizing). Then
 $\overline{\mathcal{M}}_{g,n}(X_1\times X_2,(B,0))\rightarrow\mathfrak{M}_{g,n}$ is an open substack of a stack of morphisms. In fact, this morphism is just the composition
\begin{align*}
\overline{\mathcal{M}}_{g,n}(X_1, B)\times X_2\rightarrow\overline{\mathcal{M}}_{g,n}(X_1, B)\rightarrow\mathfrak{M}_{g,n} 
\end{align*}
 of projection followed by inclusion. 
 Recall that $\overline{\mathcal{M}}_{g,n}(X_1, B)\rightarrow\mathfrak{M}_{g,n}$  has a relative obstruction theory, which is $(R\pi_*f^*TX_1)^{\vee}$, where $\pi: \mathcal{C}_{g,n}\rightarrow\overline{\mathcal{M}}_{g,n}(X_1, B)$ is the universal curve and $f:\mathcal{C}_{g,n}\rightarrow X_1$ is the universal
 stable map.
 The universal curve over $\overline{\mathcal{M}}_{g,n}(X_1\times X_2,(B,0))$ can be identified with $(\pi,id): \mathcal{C}_{g,n}\times X_2\rightarrow \overline{\mathcal{M}}_{g,n}(X_1,B)\times X_2$ and universal stable map can be identified with $(f,id): \mathcal{C}_{g,n}\times X_2\rightarrow X_1\times X_2$.  Then the relative obstruction theory of $\overline{\mathcal{M}}_{g,n}(X_1\times X_2,(B,0))\rightarrow\mathfrak{M}_{g,n}$ is $(R(\pi,id)_*(f,id)^*(TX_1\boxplus TX_2))^\vee\cong (R(\pi,id)_*(f,id)^*p_1^*TX_1)^{\vee}\oplus (R (\pi,id)_*(f,id)^*p_2^*TX_2)^{\vee}$, where we denote $p_1: X_1\times X_2\rightarrow X_1$ and $p_2: X_1\times X_2\rightarrow X_2$ to be the projection maps. Note that $R^1 (\pi,id)_*(f,id)^*p_2^*TX_2\cong \mathbb{E}^*\boxtimes TX_2$ is locally free. By the product formula for virtual fundamental class in \cite[Theorem 1]{behrend1997product}, we have 
\begin{align*}
&[\overline{\mathcal{M}}_{g,n}(X_1\times X_2,(B,0))]_{vir}=  \left([\overline{\mathcal{M}}_{g,n}(X_1, B)]_{vir}\times[X_2]\right)\cap e(\mathbb{E}^{*}\boxtimes TX_2).
\end{align*} 
\end{proof}

 Applying Proposition~\ref{prop:deg-0-vir-fun-class} to variety $Y=X\times\mathbb{P}^1$,  the virtual fundamental class of 
 the moduli space of fiber degree-0 maps to $Y=X\times\mathbb{P}^1$ has a very simple form: 
\begin{align}\label{eqn:fiber-0-Y-vir=lambda}
&[\overline{\mathcal{M}}_{g,n}(Y,(B,0))]_{vir}=  \left([\overline{\mathcal{M}}_{g,n}(X,B)]_{vir}\times[\mathbb{P}^1]\right)\cap e(\mathbb{E}^{*}\boxtimes T\mathbb{P}^1)
\nonumber\\=&\left([\overline{\mathcal{M}}_{g,n}(X,B)]_{vir}\times[\mathbb{P}^1]\right)\cap \left((-1)^g\lambda_{g}\otimes1+(-1)^{g-1}\lambda_{g-1}\otimes(2H)\right)
\end{align}
for any $g\geq0$. 

% If we change Y to be $\mathbb{P}(\omega_{X}\oplus\mathcal{O})$, fiber degree-0 virtual class???

The fiber degree-0 descendant Gromov-Witten invariants of $Y=X\times\mathbb{P}^1$ is defined to be
\begin{align}\label{eqn:fdeg=0-def-des-gw-inv-Y}
\langle\tau_{k_1}(\phi_{(\alpha_1,j_1)})...\tau_{k_n}(\phi_{(\alpha_n,j_n)})\rangle_{g,n,(B,0)}^{Y}:=\int_{[\overline{\mathcal{M}}_{g,n}(Y,(B,0))]_{vir}} \prod_{i=1}^{n}ev_i^*\phi_{(\alpha_i,j_i)}\cup c_1(L_i)^{k_i}   
\end{align}
where $L_i$ are the tautological line bundles over $\overline{\mathcal{M}}_{g,n}(Y,(B,0))$, and $ev_i: \overline{\mathcal{M}}_{g,n}(Y,(B,0))\rightarrow Y$ are
the evaluation maps. Let  $\mathbf{t}$ be the sets of variables $\{t_n^\alpha: 1\leq\alpha\leq N, n\geq0\}$ and $\mathbf{u}$ be the sets of variables $\{u_n^\alpha: 1\leq\alpha\leq N, n\geq0\}$. 
We
define the generating
function of genus-$g$ 
 fiber degree-0 descendant Gromov-Witten invariants
\begin{align}\label{eqn:def0Fggw-deg--0}
\widetilde{F}_g(\mathbf{t},\mathbf{u}):=\sum_{B\in H_2(X;\mathbb{Z})}q^{B}\left\langle\exp\left(\sum_{n,\alpha}t_n^\alpha\tau_{n}(\phi_\alpha\otimes1)+u_n^\alpha\tau_{n}(\phi_\alpha\otimes H)\right)\right\rangle_{g,(B,0)}
% =
% \sum_{n}\sum_{\substack{k_1,...,k_n\\ \alpha_1,...,\alpha_n}}\sum_{B\in H_2(X;\mathbb{Z})}\frac{q^{B}}{n!}t_{k_1}^{\alpha_1}...t_{k_n}^{\alpha_n}\langle\tau_{k_1}(\phi_{\alpha_1})...\tau_{k_n}(\phi_{\alpha_n})\rangle_{g,n,(B,0)}^{Y}. 
\end{align}
and its total  descendant potential
\begin{align}\label{eqn:total-de=0}
\widetilde{\mathcal{D}}(\mathbf{t},\mathbf{u}):=\exp\left(\sum_{g\geq0}\hbar^{2g-2}\widetilde{F}_g(\mathbf{t},\mathbf{u})\right).    
\end{align}

Then fiber degree-0 total descendant potential can be expressed in terms of Hodge integrals of $X$ as following proposition.
\begin{proposition}\label{prop:degree-0=total-Hodge}
\begin{align}\label{eqn:degree-0=total-Hodge}
&\widetilde{\mathcal{D}}(\mathbf{t},\mathbf{u})
\nonumber\\=&\exp\Bigg(\sum_{g\geq0}\hbar^{2g-2}(-1)^{g-1}2\langle\langle;\lambda_{g-1}\rangle\rangle^{X}_{g}(\mathbf{t})
+\sum_{g\geq0}\sum_{m,\beta}\hbar^{2g-2}(-1)^{g}u_{m}^{\beta}\langle\langle\tau_{m}(\phi_\beta);\lambda_{g}\rangle\rangle^{X}_{g}(\mathbf{t})\Bigg).\end{align}    
\end{proposition}
% \Xin{There should be a symplectic operator $M$, such that 
% \begin{align*}
% \widetilde{\mathcal{D}}(\mathbf{t},\mathbf{u})=\widehat{M}\mathcal{D}^{X}(\mathbf{t})  
% \end{align*}}
\begin{proof}
By equation \eqref{eqn:fiber-0-Y-vir=lambda}, 
we have
\begin{align*}
\langle\tau_{k_1}(\phi_{\alpha_1}\otimes1)...\tau_{k_n}(\phi_{\alpha_n}\otimes1)\rangle_{g,n,(B,0)}^{Y}=(-1)^{g-1} 2
\langle\tau_{k_1}(\phi_{\alpha_1})...\tau_{k_n}(\phi_{\alpha_n});\lambda_{g-1}\rangle^{X}_{g,n,B},\quad g\geq0
\end{align*}
and
\begin{align*}
\langle\tau_{k_1}(\phi_{\alpha_1}\otimes H)\tau_{k_2}(\phi_{\alpha_2}\otimes1)...\tau_{k_n}(\phi_{\alpha_n}\otimes1)\rangle^{Y}_{g,n,(B,0)}=(-1)^{g}
\langle\tau_{k_1}(\phi_{\alpha_1})...\tau_{k_n}(\phi_{\alpha_n});\lambda_{g}\rangle^{X}_{g,n,B}, \quad g\geq0. 
\end{align*}
Moreover, by degree reason, all the correlators with at least two insertions of the form $\phi_\alpha\otimes H$ vanish.
Then equation~\eqref{eqn:degree-0=total-Hodge} just follows from definition 
\eqref{eqn:def0Fggw-deg--0} and \eqref{eqn:total-de=0}. 
\end{proof}

%  The fiber degree-0 Virasoro conjectures imply relations among the latter set of descendant integrals
%  on $\overline{\mathcal{M}}_{g,n}(X,B)$.

%  \begin{lemma}
% \begin{align*}
% \langle\langle\tau_{k}(\phi_\alpha\otimes H)\rangle\rangle^{Y}_{g,bc.=0}(\mathbf{t})  =(-1)^g\langle\langle\tau_{k}(\phi_\alpha);\lambda_{g}\rangle\rangle^{X}_{g}(\mathbf{t})    
% \end{align*}
% and
% \begin{align*}
% \langle\langle\tau_{k}(\phi_\alpha\otimes1)\rangle\rangle^{Y}_{g,bc.=0}(\mathbf{t})  =(-1)^{g-1}2\langle\langle\tau_{k}(\phi_\alpha);\lambda_{g-1}\rangle\rangle^{X}_{g}(\mathbf{t})    
% \end{align*}

% \end{lemma}

% \section{The fiber degree-0 Virasoro conjecture for $Y=X\times\mathbb{P}^1$}

\subsection{The fiber degree-0 Virasoro conjecture for $Y=X\times\mathbb{P}^1$}The Virasoro 
 conjecture for $Y=X\times \mathbb{P}^1$ implies that the total descendant potential  for its fiber degree-0 descendent
 Gromov-Witten invariants is also annihilated by operators $\{L_n: n\geq-1\}$ in Proposition \ref{prop:Ln-expression}. That is 
 \begin{align*}
\widetilde{\mathcal{D}}^{-1}L_n\widetilde{\mathcal{D}}=0,\quad n\geq-1.    
 \end{align*}
 We call this  the fiber degree-0 Virasoro conjecture for $Y=X\times\mathbb{P}^1$. 
In this subsection, we establish the explicit relations between the fiber degree-0 Virasoro constraints for $Y=X\times\mathbb{P}^1$ and Conjecture~\ref{conj:lambda-g}.  
 
For the first order derivative of the total descendant potential $\widetilde{\mathcal{D}}$, we have
\begin{lemma}\label{lem:1-order-der-degree-0=total-Hodge}
For any $s$ and $\alpha$,
\begin{align*}
(i)&\quad \tau_{s}(\phi_\alpha\otimes1)\widetilde{\mathcal{D}}
\\&=\widetilde{\mathcal{D}}\cdot
\Bigg(\sum_{g\geq0}\hbar^{2g-2}(-1)^{g-1}2\langle\langle\tau_{s}(\phi_\alpha);\lambda_{g-1}\rangle\rangle^{X}_{g}
+\sum_{g\geq0}\sum_{m,\beta}\hbar^{2g-2}(-1)^{g}u_{m}^{\beta}\langle\langle\tau_{m}(\phi_\beta)\tau_{s}(\phi_\alpha);\lambda_{g}\rangle\rangle^{X}_{g}
\Bigg)
\\(ii)&\quad\tau_{s}(\phi_\alpha\otimes H)\widetilde{\mathcal{D}}   =\widetilde{\mathcal{D}}\cdot\sum_{g\geq0}\hbar^{2g-2}(-1)^{g}\langle\langle\tau_{s}(\phi_\alpha);\lambda_{g}\rangle\rangle^{X}_{g}.
\end{align*}    
\end{lemma}
\begin{proof}
This follows directly from equation~\eqref{eqn:degree-0=total-Hodge} and its derivative along vector field $\frac{\partial}{\partial t_{s}^\alpha}$ and $\frac{\partial}{\partial u_s^\alpha}$.   
\end{proof}

The action of Virasoro operator $L_n$ on total descendant potential is as follows.
\begin{theorem}
\label{thm:D-1Ln-D-formula}\begin{align*}
\widetilde{\mathcal{D}}^{-1}L_n\widetilde{\mathcal{D}}=&\sum_{g\geq0}\hbar^{2g-2}(-1)^{g}\left((-2)\Psi_{g,n}+\sum_{m,\beta}u_m^\beta\Theta_{g,n,m,\beta}\right) 
\end{align*}    
where
\begin{align}\label{eqn:Theta-formula}
&\Theta_{g,n,m,\beta}
\nonumber\\=    
&\sum_{j=0}^{n+1}\sum_{r,\alpha}e_{n+1-j}(r+b_{\alpha}-\frac{1}{2},...,r+n+b_{\alpha}-\frac{1}{2})\tilde{t}_r^{\alpha} 
\langle\langle\tau_{r+n-j}(c_1(X)^{j}\cup\phi_\alpha)\tau_{m}(\phi_\beta);\lambda_{g}\rangle\rangle^{X}_{g}
\nonumber\\&+
\sum_{j=0}^{n+1}e_{n+1-j}(m+b_{\beta}+\frac{1}{2},...,m+n+b_{\beta}+\frac{1}{2})\langle\langle \tau_{m+n-j}(c_1(X)^{j}\cup\phi_{\beta});\lambda_{g}\rangle\rangle^{X}_{g}
\nonumber\\&+\sum_{j=0}^{n+1}\sum_{s,\alpha}(-1)^{-s-1}e_{n+1-j}(-s+b^{\alpha}-\frac{3}{2},...,-s+n+b^{\alpha}-\frac{3}{2})
\nonumber\\&\quad\quad\cdot\sum_{g_1+g_2=g}\langle\langle\tau_{m}(\phi_\beta)\tau_{-s-1+n-j}(c_1(X)^j\cup\phi^\alpha);\lambda_{g_1}\rangle\rangle^{X}_{g_1}
\langle\langle\tau_{s}(\phi_\alpha);\lambda_{g_2}\rangle\rangle^{X}_{g_2}
\nonumber\\&+\delta_{g}^0\delta_{m}^{0}\sum_{\alpha}t_0^{\alpha}\int_{X}c_1(X)^{n+1}\cup\phi_\alpha\cup\phi_\beta
\end{align}
and
\begin{align*}
&\Psi_{g,n}
\nonumber\\=&\sum_{j=0}^{n+1}\sum_{r,\alpha}e_{n+1-j}(r+b_{\alpha}-\frac{1}{2},...,r+n+b_{\alpha}-\frac{1}{2})\tilde{t}_r^{\alpha} \langle\langle\tau_{r+n-j}(c_1(X)^{j}\cup\phi_\alpha);\lambda_{g-1}\rangle\rangle^{X}_{g}
\nonumber\\&-\sum_{j=0}^{n+1}\sum_{r,\alpha}(j)\cdot e_{n+1-j}(r+b_{\alpha}-\frac{1}{2},...,r+n+b_{\alpha}-\frac{1}{2})\tilde{t}_r^{\alpha}  \langle\langle\tau_{r+n-j}( c_1(X)^{j-1}\cup\phi_{\alpha});\lambda_{g}\rangle\rangle^{X}_{g}
\nonumber\\&+\sum_{j=0}^{n+1}\sum_{s,\alpha}\sum_{g_1+g_2=g}(-1)^{-s-1}e_{n+1-j}(-s-\frac{3}{2}+b^{\alpha},...,-s-\frac{3}{2}+n+b^{\alpha})
\nonumber\\&\hspace{30pt}\cdot\langle\langle\tau_{-s-1+n-j}(c_1(X)^j\cup\phi^\alpha);\lambda_{g_1-1}\rangle\rangle^{X}_{g_1}
\langle\langle\tau_{s}(\phi_\alpha);\lambda_{g_2}\rangle\rangle^{X}_{g_2}
\nonumber\\&+\frac{1}{2}\sum_{j=0}^{n+1}\sum_{s,\alpha}(-1)^{-s-1}e_{n+1-j}(-s-\frac{3}{2}+b^{\alpha},...,-s-\frac{3}{2}+n+b^{\alpha})
\nonumber\\&\hspace{60pt}\cdot\langle\langle\tau_{-s-1+n-j}(c_1(X)^j\cup\phi^\alpha)\tau_{s}(\phi_\alpha);\lambda_{g-1}\rangle\rangle^{X}_{g-1}
\nonumber\\&-\frac{1}{2}\sum_{j=0}^{n+1}\sum_{s,\alpha}\sum_{g_1+g_2=g}(j)\cdot (-1)^{-s-1}e_{n+1-j}(-s-\frac{3}{2}+b^{\alpha},...,-s-\frac{3}{2}+n+b^{\alpha})\nonumber\\&\hspace{60pt}\cdot\langle\langle\tau_{s}(\phi_{\alpha});\lambda_{g_1}\rangle\rangle^{X}_{g_1}
\langle\langle\tau_{-s-1+n-j}(c_1(X)^{j-1}\cup\phi^{\alpha});\lambda_{g_2}\rangle\rangle^{X}_{g_2}
\nonumber\\&-\delta_{g,0}\frac{n+1}{2}\sum_{\alpha,\beta}t_0^{\alpha} t_0^{\beta}\int_{X}c_1(X)^{n}\cup\phi_\alpha\cup \phi_\beta
+\frac{\delta_{g,1}\delta_{n,0}}{48}\int_{X}\left((-d)c_d(X)-2c_1(X)\cup c_{d-1}(X)\right)
\end{align*}
where $c_1(X)$ is the first Chern class of $X$, $b_\alpha=b_\alpha^{X}$ and $b^\alpha=b_X^\alpha$.
\end{theorem}
\begin{proof}
We compute the action of each terms of $L_n$ in equation~\eqref{eqn:prop-Ln-expression} on the expression of  $\widetilde{\mathcal{D}}$ in Proposition~\ref{prop:degree-0=total-Hodge} explicitly.  Formally, we write the formula~\eqref{eqn:prop-Ln-expression} of $L_n$ as $L_n=\sum_{i=1}^{8}L_n^{[i]}$, where $L_n^{[i]}$ is the $i$-th summand in the right hand side of equation~\eqref{eqn:prop-Ln-expression}. 
 For the first term in $L_n$,
\begin{align*}
L_n^{[1]}=\sum_{j=0}^{n+1}\sum_{j,\alpha}e_{n+1-j}(r+b_\alpha-\frac{1}{2},...,r+n+b_\alpha-\frac{1}{2})\tilde{t}_r^\alpha \tau_{r+n-j}((c_1(X)^j\cup\phi_\alpha)\otimes1),    
\end{align*}
then we have
\begin{align*}
&\widetilde{\mathcal{D}}^{-1}L_{n}^{[1]}\widetilde{\mathcal{D}}
\\=&\sum_{g\geq0}\hbar^{2g-2}(-1)^{g-1}2\sum_{j=0}^{n+1}\sum_{r,\alpha}e_{n+1-j}(r+b_{\alpha}-\frac{1}{2},...,r+n+b_{\alpha}-\frac{1}{2})\tilde{t}_r^{\alpha} \langle\langle\tau_{r+n-j}(c_1(X)^{j}\cup\phi_\alpha);\lambda_{g-1}\rangle\rangle^{X}_{g}
\\&+\sum_{g\geq0}\sum_{j=0}^{n+1}\sum_{r,\alpha,\beta}\sum_{m}\hbar^{2g-2}(-1)^{g}e_{n+1-j}(r+b_{\alpha}-\frac{1}{2},...,r+n+b_{\alpha}-\frac{1}{2})u_{m}^{\beta}\tilde{t}_r^{\alpha} 
\\&\hspace{80pt}\cdot
\langle\langle\tau_{r+n-j}(c_1(X)^{j}\cup\phi_\alpha)\tau_{m}(\phi_\beta);\lambda_{g}\rangle\rangle^{X}_{g}.
\end{align*}
For the second term in $L_n$,
\begin{align*}
L_n^{[2]}=\sum_{j=0}^{n+1}\sum_{r,\alpha}(2j)\cdot e_{n+1-j}(r+b_{\alpha}-\frac{1}{2},...,r+n+b_{\alpha}-\frac{1}{2})\tilde{t}_r^{\alpha} \tau_{r+n-j}(( c_1(X)^{j-1}\cup\phi_{\alpha})\otimes H),    
\end{align*}
then we have 
\begin{align*}
&\widetilde{{\mathcal{D}}}^{-1}L_n^{[2]}\widetilde{\mathcal{D}}
\\=&
\sum_{g\geq0}\hbar^{2g-2}(-1)^{g}\sum_{j=0}^{n+1}\sum_{r,\alpha}(2j)\cdot e_{n+1-j}(r+b_{\alpha}-\frac{1}{2},...,r+n+b_{\alpha}-\frac{1}{2})\tilde{t}_r^{\alpha}  \langle\langle\tau_{r+n-j}(c_1(X)^{j-1}\cup\phi_{\alpha});\lambda_{g}\rangle\rangle^{X}_{g}.
\end{align*}
For the third term in $L_n$, 
\begin{align*}
L_n^{[3]}=    
\sum_{j=0}^{n+1}\sum_{r,\alpha}e_{n+1-j}(r+b_{\alpha}+\frac{1}{2},...,r+n+b_{\alpha}+\frac{1}{2})u_r^{\alpha} \tau_{r+n-j}((c_1(X)^{j}\cup\phi_{\alpha})\otimes H),\end{align*}
then we have
\begin{align*}
&\widetilde{\mathcal{D}}^{-1}L_n^{[3]}\widetilde{\mathcal{D}}
\\=&
\sum_{g\geq0}\hbar^{2g-2}(-1)^{g}
\sum_{j=0}^{n+1}\sum_{r,\alpha}e_{n+1-j}(r+b_{\alpha}+\frac{1}{2},...,r+n+b_{\alpha}+\frac{1}{2})u_r^{\alpha}\langle\langle \tau_{r+n-j}(c_1(X)^{j}\cup\phi_{\alpha});\lambda_{g}\rangle\rangle^{X}_{g}.
\end{align*}
For the fourth term in $L_n$, we have
\begin{align*}
L_n^{[4]}= &\hbar^2\sum_{j=0}^{n+1}\sum_{s,\alpha}(-1)^{-s-1}e_{n+1-j}(-s-\frac{3}{2}+b^{\alpha},...,-s-\frac{3}{2}+n+b^{\alpha})\\&\hspace{60pt}\cdot\tau_s(\phi_{\alpha}\otimes H)\tau_{-s-1+n-j}((c_1(X)^{j}\cup\phi^{\alpha})\otimes1).
\end{align*}
By Lemma~\ref{lem:1-order-der-degree-0=total-Hodge}, it is direct to compute
\begin{align*}
&\widetilde{\mathcal{D}}^{-1}\tau_s(\phi_{\alpha}\otimes H)\tau_{-s-1+n-j}((c_1(X)^{j}\cup\phi^{\alpha})\otimes1)\widetilde{\mathcal{D}}
\\=&\left(\sum_{g\geq0}\hbar^{2g-2}(-1)^{g-1}2\langle\langle\tau_{-s-1+n-j}(c_1(X)^j\cup\phi^\alpha);\lambda_{g-1}\rangle\rangle^{X}_{g}\right)\left(
\sum_{g\geq0}\hbar^{2g-2}(-1)^{g}\langle\langle\tau_{s}(\phi_\alpha);\lambda_{g}\rangle\rangle^{X}_{g}\right)
\\&+\left(\sum_{g\geq0}\sum_{m,\beta}\hbar^{2g-2}(-1)^{g}u_{m}^{\beta}\langle\langle\tau_{m}(\phi_\beta)\tau_{-s-1+n-j}(c_1(X)^j\cup\phi^\alpha);\lambda_{g}\rangle\rangle^{X}_{g}\right)\left(
\sum_{g\geq0}\hbar^{2g-2}(-1)^{g}\langle\langle\tau_{s}(\phi_\alpha);\lambda_{g}\rangle\rangle^{X}_{g}\right)
\\&+\sum_{g\geq0}\hbar^{2g-2}(-1)^{g}\langle\langle\tau_{-s-1+n-j}(c_1(X)^j\cup\phi^\alpha)\tau_{s}(\phi_\alpha);\lambda_{g}\rangle\rangle^{X}_{g}.
\end{align*}
Then we have
\begin{align*}
&\widetilde{\mathcal{D}}^{-1}L_n^{[4]}\widetilde{\mathcal{D}}
\\=&\hbar^2\sum_{j=0}^{n+1}\sum_{s,\alpha}(-1)^{-s-1}e_{n+1-j}(-s-\frac{3}{2}+b^{\alpha},...,-s-\frac{3}{2}+n+b^{\alpha})
\\&\cdot\left(\sum_{g\geq0}\hbar^{2g-2}(-1)^{g-1}2\langle\langle\tau_{-s-1+n-j}(c_1(X)^j\cup\phi^\alpha);\lambda_{g-1}\rangle\rangle^{X}_{g}\right)
\left(\sum_{g\geq0}\hbar^{2g-2}(-1)^{g}\langle\langle\tau_{s}(\phi_\alpha);\lambda_{g}\rangle\rangle^{X}_{g}\right)
\\&+\hbar^2\sum_{j=0}^{n+1}\sum_{s,\alpha}(-1)^{-s-1}e_{n+1-j}(-s-\frac{3}{2}+b^{\alpha},...,-s-\frac{3}{2}+n+b^{\alpha})
\\&\hspace{10pt}\cdot\left(\sum_{g\geq0}\sum_{m,\beta}\hbar^{2g-2}(-1)^{g}u_{m}^{\beta}\langle\langle\tau_{m}(\phi_\beta)\tau_{-s-1+n-j}(c_1(X)^j\cup\phi^\alpha);\lambda_{g}\rangle\rangle^{X}_{g}\right)
\left(\sum_{g\geq0}\hbar^{2g-2}(-1)^{g}\langle\langle\tau_{s}(\phi_\alpha);\lambda_{g}\rangle\rangle^{X}_{g}\right)
\\&+\hbar^2\sum_{j=0}^{n+1}\sum_{s,\alpha}(-1)^{-s-1}e_{n+1-j}(-s-\frac{3}{2}+b^{\alpha},...,-s-\frac{3}{2}+n+b^{\alpha})
\\&\hspace{30pt}\cdot\left(\sum_{g\geq0}\hbar^{2g-2}(-1)^{g}\langle\langle\tau_{-s-1+n-j}(c_1(X)^j\cup\phi^\alpha)\tau_{s}(\phi_\alpha);\lambda_{g}\rangle\rangle^{X}_{g}\right).
\end{align*}
For the fifth term in $L_n$, we have
\begin{align*}
L_n^{[5]}= &\frac{\hbar^2}{2}\sum_{j=0}^{n+1}\sum_{s,\alpha}(2j)\cdot(-1)^{-s-1}e_{n+1-j}(-s-\frac{3}{2}+b^{\alpha},...,-s-\frac{3}{2}+n+b^{\alpha})\\&\hspace{60pt}\cdot\tau_s(\phi_{\alpha}\otimes H)\tau_{-s-1+n-j}( (c_1(X)^{j-1}\cup\phi^{\alpha})\otimes H).     
\end{align*}
By Lemma~\ref{lem:1-order-der-degree-0=total-Hodge}, it is direct to compute
\begin{align*}
&\widetilde{\mathcal{D}}^{-1}\tau_s(\phi_\alpha\otimes H)\tau_{-s-1+n-j}((c_1^{j-1}\cdot\phi^\alpha)\otimes H)\widetilde{\mathcal{D}}
\\=&\left(\sum_{g\geq0}\hbar^{2g-2}(-1)^{g}\langle\langle\tau_{s}(\phi_{\alpha});\lambda_{g}\rangle\rangle^{X}_{g}\right)
\left(\sum_{g\geq0}\hbar^{2g-2}(-1)^{g}\langle\langle\tau_{-s-1+n-j}( c_1(X)^{j-1}\cup\phi^{\alpha});\lambda_{g}\rangle\rangle^{X}_{g}\right).
\end{align*}
Then we have
\begin{align*}
&\widetilde{\mathcal{D}}^{-1}L_n^{[5]}\widetilde{\mathcal{D}}
\\=&\frac{\hbar^2}{2}\sum_{j=0}^{n+1}\sum_{s,\alpha}(2j)\cdot(-1)^{-s-1}e_{n+1-j}(-s-\frac{3}{2}+b^{\alpha},...,-s-\frac{3}{2}+n+b^{\alpha})\\&\hspace{30pt}\cdot\left(\sum_{g\geq0}\hbar^{2g-2}(-1)^{g}\langle\langle\tau_{s}(\phi_{\alpha});\lambda_{g}\rangle\rangle^{X}_{g}\right)
\left(\sum_{g\geq0}\hbar^{2g-2}(-1)^{g}\langle\langle\tau_{-s-1+n-j}( c_1(X)^{j-1}\cup\phi^{\alpha});\lambda_{g}\rangle\rangle^{X}_{g}\right).
\end{align*}
For the remaining three terms in $L_n$, since these are all multiplication operator, then we have
\begin{align*}
\widetilde{\mathcal{D}}^{-1}\sum_{i=6}^{8}L_n^{[i]}\widetilde{\mathcal{D}}=&\frac{n+1}{\hbar^2}\sum_{\alpha,\beta}t_0^{\alpha} t_0^{\beta}\int_{X}c_1(X)^{n}\cup\phi_\alpha\cup\phi_\beta
+\frac{1}{\hbar^2}\sum_{\alpha,\beta}t_0^{\alpha} u_0^{\beta}\int_{X}c_1(X)^{n+1}\cup\phi_\alpha\cup\phi_\beta
\nonumber\\&+\frac{\delta_{n,0}}{24}\int_{X}\left((-d)c_d(X)-2c_1(X)\cup c_{d-1}(X)\right).    
\end{align*}
 Then the  theorem follows on combining all these above formulas.
\end{proof}

Motivated by Virasoro conjecture for Gromov-Witten invariants of $Y$ and Thereom~\ref{thm:D-1Ln-D-formula}, 
we propose another conjectural universal constraints for Hodge integrals of $X$ with one $\lambda_{g-1}$ or $\lambda_{g}$ insertion.
\begin{conjecture}For any smooth projective variety $X$,
\begin{align*}
\Psi_{g,n}\equiv0\,\,  \text{for all }\, g\geq0, n\geq-1. 
\end{align*}
% \begin{align*}
% &(-1)^{g-1}2\sum_{j=0}^{n+1}\sum_{r,\alpha}e_{n+1-j}(r+b_{(\alpha,0)},...,r+n+b_{(\alpha,0)})\tilde{t}_r^{(\alpha,0)} \langle\langle\tau_{r+n-j}(c_1(X)^{j}\cdot\phi_\alpha);\lambda_{g-1}\rangle\rangle^{X}_{g}
% \\&+(-1)^{g}\sum_{j=0}^{n+1}\sum_{r,\alpha}e_{n+1-j}(r+b_{(\alpha,0)},...,r+n+b_{(\alpha,0)})2\tilde{t}_r^{(\alpha,0)}  \langle\langle\tau_{r+n-j}(j\cdot c_1(X)^{j-1}\cdot\phi_{\alpha});\lambda_{g}\rangle\rangle^{X}_{g}
% \\&+[4']\sum_{j=0}^{n+1}\sum_{s,\alpha}(-1)^{-s-1}e_{n+1-j}(-s-1+b^{(\alpha,1)},...,-s-1+n+b^{(\alpha,1)})
% \\&\sum_{g_1\geq0}\hbar^{2g_1}(-1)^{g_1-1}2\langle\langle\tau_{-s-1+n-j}(c_1(X)^j\cdot\phi^\alpha);\lambda_{g_1-1}\rangle\rangle^{X}_{g_1}
% \sum_{g_2\geq0}\hbar^{2g_2-2}(-1)^{g_2}\langle\langle\tau_{s}(\phi_\alpha);\lambda_{g_2}\rangle\rangle^{X}_{g_2}
% \\&+\sum_{j=0}^{n+1}\sum_{s,\alpha}(-1)^{-s-1}e_{n+1-j}(-s-1+b^{(\alpha,1)},...,-s-1+n+b^{(\alpha,1)})
% \\&(-1)^{g-1}\langle\langle\tau_{-s-1+n-j}(c_1(X)^j\cdot\phi^\alpha)\tau_{s}(\phi_\alpha);\lambda_{g-1}\rangle\rangle^{X}_{g-1}
% \\&+[4'']\frac{1}{2}\sum_{j=0}^{n+1}\sum_{s,\alpha}(-1)^{-s-1}e_{n+1-j}(-s-1+b^{(\alpha,1)},...,-s-1+n+b^{(\alpha,1)})\\&\hspace{60pt}\cdot\sum_{g_1+g_2=g}(-1)^{g}\langle\langle\tau_{s}(\phi_{\alpha});\lambda_{g_1}\rangle\rangle^{X}_{g_1}
% \langle\langle\tau_{-s-1+n-j}(2j c_1(X)^{j-1}\cdot\phi^{\alpha});\lambda_{g_2}\rangle\rangle^{X}_{g_2}
% \\=&
% 0
% \end{align*}   
\end{conjecture}
\subsection{Proof of Theorem~\ref{thm:virasoro-lambda-g-conj-rela} and \ref{thm:lambda-g-conj-g=0}}
Assume Virasoro conjecture for Gromov-Witten invaraints of $X$ holds, such as smooth projective varieties with semisimple quanum cohomology (cf. \cite{teleman2012structure}) or smooth algebraic curves (cf. \cite{okounkov2006virasoro}), 
then by results in \cite{MR4707267}, Virasoro conjecture for Gromov-Witten invaraints of $Y=X\times\mathbb{P}^1$ holds. In particular, the fiber degree-0 Virasoro constraints for Gromov-Witten invaraints of $Y=X\times\mathbb{P}^1$ holds, i.e. $\widetilde{\mathcal{D}}^{-1}L_n\widetilde{\mathcal{D}}\equiv0$ for any $n\geq-1$. By Theorem~\ref{thm:D-1Ln-D-formula}, this is equivalent to 
\begin{align*}
\Theta_{g,n,m,\beta}=0, \quad \text{and}\quad \Psi_{g,n}=0    
\end{align*}
for any $g,n,m,\beta$.
So we prove Theorem~\ref{thm:virasoro-lambda-g-conj-rela}. 

In genus-0, it was proven in \cite{liu1998virasoro} that Virasoro constraints for $Y=X\times\mathbb{P}^1$ always hold for any target varities $X$. In particular, the genus-0 fiber degree-0 Virasoro constraints for $Y=X\times\mathbb{P}^1$  hold.  That is $[\hbar^{-2}]\widetilde{\mathcal{D}}^{-1}L_n\widetilde{\mathcal{D}}=0, n\geq-1$. By Theorem~\ref{thm:D-1Ln-D-formula}, $\Psi_{0,n}=0$ and $\Theta_{0,n,m\beta}=0$ for any $n,m,\beta$. So we prove Theorem~\ref{thm:lambda-g-conj-g=0}. 
\section{Givental type formula for $\lambda_g$ ancestor Gromov-Witten potential}\label{sec:Givental type formula}
\subsection{Basic properties of Hodge bundle}
Recall the following  important properties of the Hodge bundle on the moduli space of stable curves (cf. \cite{mumford1983towards}):
\begin{itemize}
    \item Pulling back via the gluing map $\iota_{g_1,g_2}\colon \overline{\mathcal{M}}_{g_1,n_1+1}\times \overline{\mathcal{M}}_{g_2,n_2+1}\rightarrow \overline{\mathcal{M}}_{g,n}$, we have
    \begin{align}
    \label{eqn:prop-Hodge-1}\iota_{g_1,g_2}^*\mathbb{E}_g\cong p_1^*\mathbb{E}_{g_1}\oplus p_2^*\mathbb{E}_{g_2}\end{align}
    where $p_i (i=1,2)$ denotes via the projection maps from $\overline{\mathcal{M}}_{g_1,n_1+1}\times \overline{\mathcal{M}}_{g_2,n_2+1}$ onto its factors.
    \item Pulling back via the gluing map: $\iota_{g-1}\colon \overline{\mathcal{M}}_{g-1,n+1}\rightarrow \overline{\mathcal{M}}_{g,n}$, we have
    \begin{align}
    \label{eqn:prop-Hodge-2}\iota_{g-1}^*\mathbb{E}_g\cong \mathbb{E}_{g-1}\oplus\mathcal{O}\end{align}
    where $\mathcal{O}$ is the structure sheaf.
\end{itemize}
Then considering total Chern class $c(\mathbb{E})$ of Hodge bundle, we have
\begin{align*}
\iota_{g_1,g_2}^*c(\mathbb{E}_{g})=c(\mathbb{E}_{g_1})\cdot c(\mathbb{E}_{g_2}),\quad\text{and} \quad \iota_{g-1}^*c(\mathbb{E}_g)=c(\mathbb{E}_{g-1}).    
\end{align*}
\subsection{Semi-simple cohomological field theories}
In this section, we review Givental-Teleman classification for semi-simple cohomological field theories and refer the reader to \cite{givental2001semisimple}, \cite{teleman2012structure} and \cite{pandharipande2015relations} for more details.

A  stable graph $\Gamma=(V,H,E,L,g,p,\iota,m)$ of genus $g$ with $n$ legs consists of the following data:
\begin{enumerate}
\item a finite set of vertices $V$  with a genus function $g\colon V\to \mathbb{Z}_{\geq 0}$;

\item a finite set of half-edges $H$  with a vertex assignment $p\colon H\to V$ and an involution $\iota\colon H\to H$;

\item a set of edges $E$, which is the set of two-point orbits of $\iota$;

\item a set of legs $L$, which is the set of fixed points of $\iota$, and which is marked by a bijection $m\colon\{1,\ldots,n\}\to L$.

\item The graph $(V,E)$ is connected.

\item For every vertex $v\in V$, we have
\begin{equation*}
  2g(v)-2+n(v) > 0,
\end{equation*}
where $n(v)=|  p^{-1}(v)|$ is the {\it valence} of the vertex $v$
\item The genus of the graph is $g$, in the sense that
\begin{equation*}\label{eq:genus}
g=h^1(\Gamma)+\sum_{v\in V}g(v),
\end{equation*}
where $h^1(\Gamma)= |E|- |V|+1$.
\end{enumerate}

An automorphism of a stable graph $\Gamma$ consists of permutations of the sets $V$ and $H$ which
leave invariant the structures  $g$, $p$, $\iota$ and $m$ (and hence preserve $L$ and $E$). We denote by $\Aut(\Gamma)$ the automorphism group of $\Gamma$.

Recall that a cohomological field theory (CohFT), first defined by
 Kontsevich and Manin (cf. \cite{kontsevich1994gromov}), consists of a finite-dimensional $\mathbb{C}$-vector
 space $V$ equipped with a non-degenerate pairing $\eta$, a distinguished element $\mathbf{1}\in V$, and a system of homomorphisms $\Omega_{g,n}: V^{\otimes n}\rightarrow H^*(\overline{\mathcal{M}}_{g,n};\mathbb{C})$ satisfying a number of compatibility axioms.  A
 quantum product $*$ on $V$ can be defined for any CohFT yields  by $\eta(v_1*v_2,v_3)=\Omega_{0,3}(v_1,v_2,v_3)$. We call  a CohFT  semisimple if there exists a canonical basis $\{e_i\}$ of $V$ such that $e_i*e_j=\delta_{ij}e_i$ and $\sum_{i}e_i=\mathbf{1}$. 

 The (upper half of the) symplectic loop group corresponding to a vector space
 $V$ with non-degenerate pairing $\eta$ is the group of endomorphism valued power
 series $V[[z]]$ such that the symplectic condition $R(z)R^*(-z)=Id$ holds. Here $R^*(z)$ is the adjoint of $R(z)$ with respect to $\eta$. 

In 2012, Teleman proved that a semisimple
 CohFT $\Omega$ can be reconstructed from its topological part $\omega$ via an $R$ matrix action,  where $\omega$ is the projection of $\Omega$ to $H^0(\overline{\mathcal{M}}_{g,n})$ (cf. \cite{teleman2012structure}). More precisely, for any vector $w_1,..,w_n\in V$,
 \begin{align}\label{eqn:tele-recontr}
 \Omega_{g,n}(w_1,...,w_n)=\sum_{\Gamma\in G_{g,n}}\frac{1}{|\Aut(\Gamma)|}\xi_{\Gamma,*}\left(\prod_{v}\sum_{k=0}^{\infty}\frac{1}{k!}\pi_{*}\omega_{g(v),n(v)+k}(\cdots)\right)    
 \end{align}
  where  $G_{g,n}$ is the set of  stable  graphs of genus $g$ with $n$ legs, and $\xi_{\Gamma}:\prod_{v}\overline{\mathcal{M}}_{g(v),n(v)}\rightarrow \overline{\mathcal{M}}_{g,n}$ is the gluing map of curves of topological
 type $\Gamma$ from their irreducible components, $\pi: \overline{\mathcal{M}}_{g(v),n(v)+k}\rightarrow \overline{\mathcal{M}}_{g(v),n(v)}$
 forgets the last $k$ markings. Then  we specify what is to put into the
 insertions of $\omega_{g(v),n(v)+k}$. Instead of allowing only vectors in $V$ to be
 put into $\Omega_{g,n}$, we will allow elements of $V[[\psi_1,...,\psi_n]]$, where $\psi_i$ acts on
 the cohomology of the relevant moduli space of curves by multiplication
 with the $i$-th cotangent line class.
 \begin{enumerate}
     \item  Into each insertion corresponding to $i$-th marking, put
 $R^{-1}(\psi_i)$ applied to the corresponding vector $w_i$.
 \item  Into each pair of insertions corresponding to an edge put the bi-vector 
\[\frac{Id\otimes Id-R^{-1}(\psi_1)\otimes R^{-1}(\psi_2)}{\psi_1+\psi_2}\eta^{-1}\]
 where one has to substitute the $\psi$-classes at each side of the normalization of the node for $\psi_1$ and $\psi_2$. By the symplectic condition this
 is well-defined.
 \item At each of the additional arguments for each vertex put
 \[T(\psi)=\psi(Id-R^{-1}(\psi))\mathbf{1}\]
  where $\psi$ is the cotangent line class corresponding to the additional marking.
 Since $T(z) = O(z^2)$ the above $k$-sum is finite.
 \end{enumerate}
\subsection{Reconstruction formula for $\lambda_{g}$ invariants of varieties at semisimple point}\label{subsec:recons-lambda-g-invariants}
Consider the stablization morphism
\[\St\colon \overline{\mathcal{M}}_{g,n}(X,B)\rightarrow \overline{\mathcal{M}}_{g,n}\]
which forgets the map $f$ in $(C; x_1,\dots,x_n; f ) \in  \overline{\mathcal{M}}_{g,n}(X,B)$ and stabilizes the curve
$(C; x_1,\dots,x_n)$
by contracting unstable components to points. 
Let
\[\pi_k\colon \overline{\mathcal{M}}_{g,n+k}(X,B)\rightarrow \overline{\mathcal{M}}_{g,n}(X,B)\]
be the morphism, which forgets the last $k$ markings, and stabilizes the curve.
For $t\in H^*(X;\mathbb{C})$, we define a CohFT associated to Gromov-Witten invariants of $X$
\begin{align*}
\Omega_{g,n}^{t}(\phi_{\alpha_1},...,\phi_{\alpha_n}):=\sum_{k\ge0}\frac{1}{k!}\sum_{B\in H_2(X;\mathbb{Z})}q^B\St_{*}\pi_{k*}\left(\prod_{i=1}^{n}ev_i^*\phi_{\alpha_i}
\prod_{i=n+1}^{n+k}ev_i^*t\cap[\overline{\mathcal{M}}_{g,n+k}(X,B)]_{vir}\right).    
\end{align*}

If we times $\lambda_{g}$ class, we get the following reconstruction theorem  
\begin{theorem}\label{thm:lambda-g-times-CohFT}
If $t\in H^*(X;\mathbb{C})$ is a semisimple point of the  CohFT associated to Gromov-Witten theory of $X$, there exist $R$ matrix, such that   
for any vector $w_1,..,w_n\in H^*(X;\mathbb{C})$,
 \begin{align}
\label{eqn:lambda-g-times-Omega} \lambda_{g}\cdot\Omega^{t}_{g,n}(w_1,...,w_n)=\sum_{\Gamma\in G_{g,n}^{c}}\frac{1}{|\Aut(\Gamma)|}\xi_{\Gamma,*}\left(\prod_{v}\sum_{k=0}^{\infty}\frac{1}{k!}\pi_{k*}\left(\lambda_{g(v)}\cdot\omega^{t}_{g(v),n(v)+k}(\cdots)\right)\right)    
 \end{align}
 where $G_{g,n}^{c}\subset G_{g,n}$ is the set of  stable graphs which is a tree.  $\omega^{t}$ is the topological part of $\Omega^{t}$. All the insertions $(...)$ are the same as in equation~\eqref{eqn:tele-recontr}.
\end{theorem}
\begin{proof}
This just follows from \eqref{eqn:tele-recontr}, and properties of Hodge bundle~\eqref{eqn:prop-Hodge-1} and \eqref{eqn:prop-Hodge-2}.      
\end{proof}

For $\mathbf{t}=\sum_{n,\alpha}t_n^\alpha z^n\phi_\alpha\in H^*(X;\mathbb{C})[[\{t_n^\alpha\}]][[z]]$, we define 
  genus $g$ ancestor potential of $\lambda_{g}$ invariants at base point $t$ by
  \begin{align*}
  \bar{\mathcal{F}}_{g}^{\lambda_{g}}(t,\mathbf{t})=\sum_{n>2-2g}\sum_{k\geq0}\sum_{B\in H_2(X;\mathbb{Z})}\frac{q^B}{k!}\int_{[\overline{\mathcal{M}}_{g,n+k}(X,d)]_{vir}}\lambda_{g}\prod_{i=1}^{n}ev_i^*\mathbf{t}(\bar{\psi}_i)\prod_{i=n+1}^{n+k}ev_i^{*}t
  \end{align*}
  where $\bar{\psi}_i=\pi_k^*\St^*c_1(L_i)$ with $L_i$ is the cotangent line bundle along the $i$-th marked point over $\overline{\mathcal{M}}_{g,n}$. 
 The total ancestor potential of $\lambda_{g}$ invariants that depends on the choice of $t$ is defined by 
	\begin{align}
	\mathcal{A}_{t}^{\lambda}(\hbar,\mathbf{t})=\exp\left(\sum_{g\geq0}\hbar^{2g-2}\bar{\mathcal{F}}_{g}^{\lambda_{g}}(t,\mathbf{t})\right).  
	\end{align}
	% where $\mathbf{t}(\psi)=\sum_{k\geq0}t_k\psi^k$, where  and $t_k (k\geq0)$ are formal variables, and where
	% $\psi_i (i=1,...,n)$ is the 1st Chern class of the orbifold line bundle on $\overline{M}_{g,n}$ formed
	% by the cotangent lines at the $i$-th marked point. 

In below, we would like use Theorem~\ref{thm:lambda-g-times-CohFT} to reconstruct the total ancestor  potential $\mathcal{A}_{t}^{\lambda}(\hbar,\mathbf{t})$ of $\lambda_{g}$ invariants in terms of $\lambda_{g}$ invariants of a point.  
 
For convenience, we denote the non-degenerate intersection paring $\eta$ on $H^*(X;\mathbb{C})$ by $\langle\,,\,\rangle$.	If $t\in H^*(X;\mathbb{C})$ is a semisimple point of the associated CohFT, assume $\{e_i=e_i(t)\}_{i=1}^{N}$ is the canonical basis and $\{e^i=e^i(t)\}_{i=1}^{N}$ is their dual basis and we denote the  norm of canonical basis by ${\Delta_i}^{-1}=\langle e_i,e_i\rangle$. 
	Given a stable tree graph $\Gamma\in G_{g,n}^{c}$, we introduce  a set $LA(\Gamma)$ which consists of the following labels:
	\begin{enumerate}
		\item (marking) $i: V(\Gamma) \to \{1,\cdots,N\}$. This induces
		$i:H(\Gamma)\to \{1,\cdots,N\}$, as follows:
		if $h\in H(\Gamma)$ is a leaf attached to a vertex $v\in V(\Gamma)$,
		define $i(h)=i(v)$.
		\item (height) $a: H(\Gamma)\to \mathbb{Z}_{\geq 0}$.
	\end{enumerate}
     
	Then we assign contributions to legs, edges, and vertices of a stable tree graph $\Gamma\in G^c_{g,n}$ as follows.
\begin{enumerate}
		\item {\em Legs.} To each leg $l \in L(\Gamma)$ with  $i(l)= i\in \{1,\cdots,N\}$
		and  $a(l)= a\in \mathbb{Z}_{\geq 0}$, place:
	\[[z^{a(l)}]\langle R(z)^{-1}\mathbf{t},e^{i(l)}\rangle.\]
 Here we use $[z^a]$ to denote taking the coefficient of $z^a$.
		\item {\em Edges.} To an edge $e\in E(\Gamma)$ connected a vertex marked by $i\in \{1,\cdots,N\}$ to a vertex
		marked by $j\in \{1,\cdots,N\}$ and with heights $a$ and $b$ at the corresponding half-edges, place
		$$
		[z^{a(h)} w^{a(h')}]\sum_{\alpha}\left(\frac{1}{z+w}\left(\langle \phi_\alpha,e^{i(h)}\rangle\langle \phi^\alpha,e^{i(h')}\rangle-\langle R(z)^{-1}\phi_\alpha,e^{i(h)}\rangle\langle R(w)^{-1}\phi^\alpha,e^{i(h')}\rangle\right)\right),$$
  	where $h,h'$ are the two half edges associated to $e$. In fact, the order of the two half edges does not affect the graph sum formula. 
		\item {\em Vertices.} To a vertex $v$ with genus $g(v)$,	marking $i(v)$, 
 and $n(v)$  half-edges attached to it with heights $a_1, ..., a_{n(v)} \in \mathbb{Z}_{\geq 0}$, place
\begin{align*}
&{\Delta_{i(v)}}^{g(v)-1}\sum_{k\geq0}\frac{1}{k!}\sum_{b_1,...,b_k\geq2}
 \left(\prod_{j=1}^{k}[z^{b_j}]\langle z\left(Id-R(z)^{-1}\right)1,e^{i(v)}\rangle\right)
 \\&\hspace{60pt}\cdot\int_{\overline{\mathcal{M}}_{g(v),n(v)+k}}
\lambda_{g(v)}\prod_{\{h:p(h)=v\}}\psi_{h}^{a(h)}
\prod_{j=1}^{k}\psi_{n(v)+j}^{b_j}.    
\end{align*}
By degree reason, this is actually a finite sum over $m$. 	\end{enumerate}
 \begin{corollary}\label{cor:total-ancestor-potential-lambda-g}
At a semisimple point $t\in H^*(X;\mathbb{C})$, the total ancestor potential of $\lambda_{g}$ invariants has the following graph sum formula
\begin{small}
	\begin{align}	\label{eqn:ancestor-graph-sum-formula}&\mathcal{A}^{\lambda}_t(\hbar,\mathbf{t}) 
\nonumber \\=& \exp\Bigg(\sum_{2g-2+n>0}\frac{\hbar^{2g-2}}{n!} \sum_{\Gamma\in G^{c}_{g,n}}\sum_{(a,i)\in LA(\Gamma)}\frac{1}{|\Aut(\Gamma)|}
 \nonumber\\&\hspace{30pt}\prod_{v}{\Delta_{i(v)}}^{g(v)-1}\sum_{k\geq0}\frac{1}{k!}\sum_{b_1,...,b_k\geq2}
 \left(\prod_{j=1}^{k}[z^{b_j}]\langle z\left(Id-R(z)^{-1}\right)1,e^{i(v)}\rangle\right)
 \nonumber\\&\hspace{160pt}\cdot\int_{\overline{\mathcal{M}}_{g(v),n(v)+k}}
\lambda_{g(v)}\prod_{\{h:p(h)=v\}}\psi_{h}^{a(h)}
\prod_{j=1}^{k}\psi_{n(v)+j}^{b_j} 
 \nonumber\\&\hspace{25pt}\prod_{e=(h,h')}[z^{a(h)} w^{a(h')}]\sum_{\alpha}\left(\frac{1}{z+w}\left(\langle \phi_\alpha,e^{i(h)}\rangle\langle \phi^\alpha,e^{i(h')}\rangle-\langle R(z)^{-1}\phi_\alpha,e^{i(h)}\rangle\langle R(w)^{-1}\phi^\alpha,e^{i(h')}\rangle\right)\right)
 \nonumber\\&\hspace{30pt}\prod_{l}[z^{a(l)}]\langle R(z)^{-1}\mathbf{t},e^{i(l)}\rangle\Bigg)
	.\end{align}
    
\end{small}
     
 \end{corollary}
 \begin{proof}
 For each element $t\in H^*(X;\mathbb{C})$, it is useful to introduce the ancestor  double bracket notation:
	\begin{align*}
	\langle\langle\phi_{\alpha_1}\bar{\psi}_1^{k_1},...,\phi_{\alpha_n}\bar{\psi}_n^{k_n}
	\rangle\rangle_{g,n}(t):=\sum_{k\geq0}\sum_{B\in H_2(X;\mathbb{Z})}\frac{q^B}{k!}\int_{[\overline{\mathcal{M}}_{g,n+k}(X,B)]_{vir}}\lambda_{g}\prod_{i=1}^{n}ev_i^*\mathbf{t}(\bar{\psi}_i)\prod_{i=n+1}^{n+k}ev_i^{*}t.    
	\end{align*}
where $\bar{\psi}_i=\pi_k^*\St^*c_1(L_i)$ with $L_i$ is the cotangent line bundle along the $i$-th marked point over $\overline{\mathcal{M}}_{g,n}$.

 By definition of CohFT $\{\Omega^t_{g,n}\}$, it is easy to see
 \begin{align}\label{eqn:double-braket-int-lamb-psi-Omega} \langle\langle\phi_{\alpha_1}\bar{\psi}_1^{k_1},...,\phi_{\alpha_n}\bar{\psi}_n^{k_n}	\rangle\rangle_{g,n}(t)=\int_{\overline{\mathcal{M}}_{g,n}}\lambda_{g}\prod_{i=1}^{n}\psi_{i}^{k_i}\cdot\Omega^t_{g,n}(\phi_{\alpha_1},...,\phi_{\alpha_n})    
 \end{align}
and  the total ancestor potential of $\lambda_{g}$ invariants
 \begin{align}
\label{eqn:ancestor-exp-}	\mathcal{A}_{t}^{\lambda}(\hbar,\mathbf{t})=\exp\left(\sum_{2g-2+n>0}\frac{\hbar^{2g-2}}{n!}\langle\langle\mathbf{t}(\bar{\psi}_1),...,\mathbf{t}(\bar{\psi}_n);\lambda_{g}\rangle\rangle_{g,n}(t)\right).  
	\end{align}
For the topological field theory part $\omega^t$ of CohFT $\Omega^{t}$, via the axioms of CohFT, it is easy to compute
\begin{align}\label{eqn:omega-t-e-...-e}
\omega^t_{g,n}(e_{i_1},...,e_{i_n})
=
\begin{cases}
{\Delta_{i}}^{g-1},\, i_1=...=i_n=i\\
0,\, \quad\quad\text{others}
\end{cases}
\end{align}
Then the graphs sum formula of ancestor potential~\eqref{eqn:ancestor-graph-sum-formula} follows from equations~\eqref{eqn:lambda-g-times-Omega}, \eqref{eqn:double-braket-int-lamb-psi-Omega}, \eqref{eqn:ancestor-exp-} and \eqref{eqn:omega-t-e-...-e}.    
 \end{proof}
% \begin{proof}
% \begin{align*}
% 	&	{\Delta_{i(v)}}^{g(v)-1}\sum_{m\geq0}\frac{1}{m!}\sum_{b_1,...,b_m\geq2}
%  \left(\prod_{j=1}^{m}[z^{b_j}]\langle z\left(Id-R(z)^{-1}\right)1,e^i\rangle\right)\int_{\overline{\mathcal{M}}_{g(v),n(v)+m}}\prod_{j=1}^{n(v)}\psi_j^{a_j} \prod_{j=1}^{m}\psi_{n(v)+j}^{b_j}
% \lambda_{g(v)}	\\=&{\Delta_{i(v)}}^{g(v)-1}\sum_{m\geq0}\frac{1}{m!}\sum_{b_1,...,b_m\geq2}
%  \left(\prod_{j=1}^{m}[z^{b_j}]\langle z\left(Id-R(z)^{-1}\right)1,e^i\rangle\right)\binom{2g(v)-3+n(v)+m}{a_1,...,a_{n(v)},b_1,...,b_m}b_{g(v)}	.
% 		\end{align*}   
% By results in \cite{faber2000hodge}, 
% \[b_{g}=\int_{\overline{\mathcal{M}}_{g,1}}\psi_{1}^{2g-2}\lambda_{g}=\frac{2^{2g-1}-1}{2^{2g-1}}\frac{|B_{2g}|}{(2g)!},\quad g\geq1.\]

% \end{proof}	

The main difference between Givental formula for ancestor Gromov-Witten potential of $X$ and ancestor  potential of $\lambda_{g}$ invariants is the type of the graphs in the sum: the former one is over all possible stable graphs, while the latter one is required to be stable trees. The second difference is about the vertex contributions with explicit formula for $\lambda_{g}$ integrals (cf. \cite{faber2000hodge}). This will make it much simpler to compute $\lambda_{g}$ invariants than pure ancestor Gromov-Witten invariants.

It is remarkable that this type of graph sum formula in Theorem~\ref{thm:lambda-g-times-CohFT}  also appears in \cite{arsie2023semisimple}, which is defined for more general semisimple  F-CohFTs. It is very natural to ask whether $\lambda_{g}$ conjecture holds true for general semisimple  F-CohFTs or semisimple flat F-manifolds. We postpone to study  this interesting problem in the  future.

\section{New univeral constraints for desendant Gromov-Witten invariants}
\label{sec:new-universal-gw-pixconj}

\subsection{Pixton's formula for double ramification cycles}
% Denote $G_{g,n}$ to be the set of stable graphs of genus-$g$, $n$ marked points \cite[\S4.2]{pandharipande2015calculus}.
Let $A=(a_1,\ldots,a_n)$ be a vector of double ramification data,
i.e.\ a vector of integers such that $\sum_{i=1}^na_i=0$.
Let $\Gamma \in G_{g,n}$ be a stable graph of genus $g$ with $n$ legs
and $r$ be a positive integer.
A weighting mod $r$ of $\Gamma$ is a function on the set of
half-edges,
$$ w\colon H(\Gamma) \rightarrow \{0,\dots,r-1\},$$
which satisfies the following three properties:
\begin{enumerate}
\item[(i)] $\forall h_i\in L(\Gamma)$, corresponding to
 the marking $i\in \{1,\ldots, n\}$,
$$w(h_i) \equiv a_i \mod r \ ,$$
\item[(ii)] $\forall e \in E(\Gamma)$, corresponding to two half-edges
$h,h' \in H(\Gamma)$,
$$w(h)+w(h') \equiv 0 \mod r\, ,$$
\item[(iii)] $\forall v\in V(\Gamma)$,
$$\sum_{v(h)= v} w(h) \equiv 0 \mod r\, ,$$ 
where the sum is taken over all  half-edges incident to $v$.
\end{enumerate}
Let $\mathsf{W}_{\Gamma,r}$ be the set of weightings mod $r$
of $\Gamma$. The set $\mathsf{W}_{\Gamma,r}$ is finite, with cardinality $r^{h^1(\Gamma)}$.
For each stable graph $\Gamma\in G_{g,n}$, we associate  a moduli space $\overline{\mathcal{M}}_\Gamma:=\prod_{v\in V(\Gamma)}\overline{\mathcal{M}}_{g(v),n(v)}$ and define a natural map $\xi_{\Gamma}: \overline{\mathcal{M}}_\Gamma \to 
\overline{\mathcal{M}}_{g,n}$ to be the  canonical gluing morphism.

We denote by
$\mathsf{P}_g^{d,r}(A)\in R^d(\overline{\mathcal{M}}_{g,n})$ the degree $d$ component of the tautological class 
\begin{align*}
\hspace{-5pt}\sum_{\Gamma\in {G}_{g,n}} 
\sum_{w\in \mathsf{W}_{\Gamma,r}}
\frac{1}{|\Aut(\Gamma)| } 
\frac{1}{r^{h^1(\Gamma)}}
\xi_{\Gamma*}\Bigg[
\prod_{i=1}^n \exp(a_i^2 \psi_{h_i}) \cdot 
\prod_{e=(h,h')\in E(\Gamma)}
\frac{1-\exp(-w(h)w(h')(\psi_h+\psi_{h'}))}{\psi_h + \psi_{h'}} \Bigg]
\end{align*} 
in $R^*(\overline{\mathcal{M}}_{g,n})$.

Pixton proved for fixed $g$, $A$, and $d$, the \label{pply}
class
$\mathsf{P}_g^{d,r}(A) \in R^d(\overline{\mathcal{M}}_{g,n})$
is polynomial in $r$ for $r$ sufficiently large. 
We denote by $\mathsf{P}_g^d(A)$ the value at $r=0$ 
of the polynomial associated to $\mathsf{P}_g^{d,r}(A)$. In other words, $\mathsf{P}_g^d(A)$ is the  constant term of the associated polynomial in $r$. 

The following formula for double ramification cycles was first conjectured
by Pixton.
\begin{theorem}[\cite{janda2017double}] \label{thm:DR-Pixton}
For $g\geq 0$ and double ramification data $A$, we have
$$\mathsf{DR}_g(A) = 2^{-g}\, \mathsf{P}_g^g(A)\, \in R^g(\overline{\mathcal{M}}_{g,n}).$$ 
\end{theorem}

The top Chern class $\lambda_g$ of the Hodge bundle $\mathbb{E}$ is a very
special case of a double ramification cycle (see
\cite[\S0.5.3]{janda2017double}):
\begin{equation}
\label{eqn:DR-lambda-g}\mathsf{DR}_g(0,\ldots,0) = (-1)^g\lambda_g \ \in R^g(\overline{\mathcal{M}}_{g,n}).\    
\end{equation}

Combining Theorem~\ref{thm:DR-Pixton} and \eqref{eqn:DR-lambda-g}, it
leads to a nice formula for the top Chern class of Hodge bundle
$\lambda_g$, which is supported on the boundary divisor of curves with
a non-separating node.
That is, in $R^{g}(\overline{\mathcal{M}}_{g,n})$, we have
\begin{equation}\label{eqn:lambda-g-formula}
  \lambda_g=\frac{(-1)^g}{2^g} \mathsf{P}_{g}^{g}(0,\dots,0) .
\end{equation} 
\subsection{Descendant and ancestor correspondence }
We first recall some standard results about the correspondence between descendant and ancestor invariants (cf. \cite[Appendix 2]{coates2007quantum}). 

For $1\leq i\leq n$, denote the  cotangent line bundle  along the $i$-th marked point over $\overline{\mathcal{M}}_{g,n}$ and $\overline{\mathcal{M}}_{g,n+m}(X,B)$ by $L_i$ and $\mathcal{L}_i$, respectively.
The bundles $\mathcal{L}_i$ and $\pi_m^*\St^*L_i$  over
$\overline{\mathcal{M}}_{g,n+m}(X,B)$ are identified outside the locus $D$ consisting of maps such that the $i$-th marked
point is situated on a component of the curve which gets collapsed by $\St \circ \pi_m$.
This locus $D$ is the image of the gluing map
\begin{align*}
   \iota\colon: \sqcup_{\substack{m_1+m_2=m\\ B_1+B_2=B}} \overline{\mathcal{M}}_{0,\{i\}+\bullet+m_1}(X,B_1)\times_{X} \overline{\mathcal{M}}_{g,[n]\setminus\{i\}+\circ+m_2}(X,B_2)\rightarrow \overline{\mathcal{M}}_{g,n+m}(X,B)
\end{align*}
where the two markings $\bullet$ and $\circ$ are glued together under map $\iota$.
We denote the domain of this map by $Y_{n,m,B}^{(i)}$. The virtual normal bundle to $D$ at a generic
point is $\Hom(\pi_m^*\St^*c_1(L_i), \mathcal{L}_i)$, and so $D$ is ``virtually Poincar\'e-dual'' to $c_1(\mathcal{L}_i)-\pi_m^*\St^*c_1(L_i)$ in the sense that
\begin{align}
\label{eqn:D-Y-vir}  \left(c_1(\mathcal{L}_i)-\pi_m^*\St^*c_1(L_i)\right) \cap [\overline{\mathcal{M}}_{n+m}{(X,B)}]^{\vir}=\iota_*[Y_{n,m,B}^{(i)}]^{\vir}
\end{align}

Define mixed type of ancestor and descendant  twisted correlator as follows:
\begin{align}
&\langle\langle\bar{\tau}_{j_1}\tau_{i_1}(\phi_1),\dots,\bar{\tau}_{j_n}\tau_{i_n}(\phi_n);c_{g,n}\rangle\rangle_g(\mathbf{t}(\psi))
\nonumber\\&:=\sum_{m\geq0}\sum_{B\in H_2(X;\mathbb{Z})}\frac{q^B}{m!}
\int_{[\overline{\mathcal{M}}_{g,n+m}(X,B)]^{vir}} \prod_{k=1}^{n}\left(\bar{\psi}_{k}^{j_k}\psi_k^{i_k}\ev_k^*\phi_k\right)\prod_{k=n+1}^{n+m}\ev_k^*\mathbf{t}(\psi_{k})\cdot \pi_m^* \St^*c_{g,n}.
\end{align}
for certain cohomology class $c_{g,n}\in H^*(\overline{\mathcal{M}}_{g,n})$, where $\psi_k:=c_1(\mathcal{L}_k)$ and $\bar{\psi}_k:=\pi_m^*\St^*c_1(L_k)$. 
Below we only focus the cases of $c_{g,n}=1$ or $\lambda_g$. 

Recall 
 the following $T$ operator which was studied in
\cite{liu2002quantum}
\begin{equation}
  \label{eq:T-operator}
  T(\mathcal{W}):=\tau_+(\mathcal{W})-\sum_{\alpha}\langle\langle \mathcal{W}\phi^\alpha\rangle\rangle_0\phi_\alpha
\end{equation}
for any vector field $\mathcal{W}$ on the big phase space, where $\tau_+(\mathcal{W})$  is a linear operator defined
by $\tau_{+}(\tau_n(\phi_\alpha))=\tau_{n+1}(\phi_\alpha)$. 

The following lemma was proved in \cite{janda2024universal}:
\begin{lemma}\label{lem:anc-desc}
For any $2g-2+n>0$,  on the big phase space, we have
\begin{align*}
% \label{eqn:anc-des}
\langle\langle\bar\tau_{j_1}\tau_{i_1}(\phi_{\alpha_1}),\dots,\bar\tau_{j_n}\tau_{i_n}(\phi_{\alpha_n});\lambda_g\rangle\rangle_g =   \langle\langle  T^{j_1}(\tau_{i_1}(\phi_{\alpha_1})),\dots,  T^{j_n}(\tau_{i_n}(\phi_{\alpha_n}));\lambda_g\rangle\rangle_g 
\end{align*}
for any $\{\phi_{\alpha_i}: i=1,\dots,n\}\subset H^*(X;\mathbb{C})$ and $T$ is the operator from \eqref{eq:T-operator}.
In particular,
\begin{align}\label{eqn:anc-des-2}
\langle\langle\bar\tau_{j_1}(\phi_{\alpha_1}),\dots,\bar\tau_{j_n}(\phi_{\alpha_n});\lambda_g\rangle\rangle_g =   \langle\langle  T^{j_1}(\phi_{\alpha_1}),\dots,  T^{j_n}(\phi_{\alpha_n});\lambda_g\rangle\rangle_g. 
\end{align}
Note that by \eqref{eqn:lambda-g-formula}, these two equations also hold true if we replace $\lambda_g$ by $\mathsf{P}_g^g(0,..,0)$. 
\end{lemma}
% \begin{proof}
%   We will focus on the first marked point, and for simplicity, suppress the content of the other marked
%   points from our notation.
%   By \eqref{eqn:D-Y-vir} and \eqref{eqn:prop-Hodge-1}, the following identity holds for mixed type twisted correlators on the big phase space
% \begin{align*}
% \langle\langle\bar{\tau}_{j_1-1}\tau_{i_1+1}(\phi),\dots;\lambda_g\rangle\rangle_g=
% \langle\langle\bar{\tau}_{j_1}\tau_{i_1}(\phi),\dots;\lambda_g\rangle\rangle_g
% +\langle\langle \tau_{i_1}(\phi),\phi_\beta\rangle\rangle_0\langle\langle\bar{\tau}_{j_1-1}(\phi^\beta),\dots;\lambda_g\rangle\rangle_g
% \end{align*}
% for any $\phi\in H^*(X;\mathbb{C})$.

% This implies
% \begin{align*}
% &\langle\langle\bar{\tau}_{j_1}\tau_{i_1}(\phi_{\alpha_1}),\dots;\lambda_g\rangle\rangle_g 
% =\langle\langle\bar\tau_{j_1-1}\left(\tau_{i_1+1}(\phi_{\alpha_1})-\langle\langle\tau_{i_1}(\phi_{\alpha_1}),\phi_\beta\rangle\rangle_0\phi^\beta\right),\dots;\lambda_g\rangle\rangle_g
% \\=&\langle\langle\bar\tau_{j_1-1}T(\tau_{i_1}(\phi_{\alpha_1})),\dots;\lambda_g\rangle\rangle_g.
% \end{align*}
% Repeating this $j_1 - 1$ more times, we get
% \begin{align*}
%   \langle\langle\bar{\tau}_{j_1}\tau_{i_1}(\phi_{\alpha_1}),\dots;\lambda_g\rangle\rangle_g 
%   =\langle\langle T^{j_1}(\tau_{i_1}(\phi_{\alpha_1})),\dots;\lambda_g\rangle\rangle_g
% \end{align*}
% Similar analysis for the other markings, yields \eqref{eqn:anc-des}.
% \end{proof}

For any $g, n$ such that $2g - 2 + n > 0$ and $\mathbf{t}=\sum_{n,\alpha}t_n^\alpha z^n\phi_\alpha\in H^*(X)[[\{t_n^\alpha\}]][[z]]$, define a homomorphism $\Omega^{\mathbf{t}}_{g,n}\colon H^*(X)[[\psi]]^{\otimes n}\rightarrow H^*(\overline{\mathcal{M}}_{g,n})[[q]][[\{t_n^\alpha\}]]$ via
\begin{align*}
\Omega^{\mathbf{t}}_{g,n}(\phi_{\alpha_1}\psi_1^{i_1},\dots,\phi_{\alpha_n}\psi_n^{i_n})  
=\sum_{m\geq0}\sum_{B\in H_2(X;\mathbb{Z})}\frac{q^B}{m!}\St_*{\pi_m}_*\left(\prod_{k=1}^{n}\ev_k^*\phi_{\alpha_k}\psi_k^{i_k}\prod_{k=n+1}^{n+m}\ev_i^*\mathbf{t}(\psi_{k})\right).
\end{align*}
From the ``cutting edges'' axiom of virtual cycles of moduli of stable maps to $X$, this system of homomorphisms satisfies the splitting axioms:
\begin{align}\label{eqn:split-ax-1}
\iota_{g_1,g_2}^*\Omega_{g,n}^{\mathbf{t}}(\phi_{\alpha_1}\psi_1^{i_1},\dots,\phi_{\alpha_n}\psi_{n}^{i_n})
=\sum_{\beta}\Omega_{g_1,n_1+1}^{\mathbf{t}}(\otimes_{k\in I_1}\phi_{\alpha_k}\psi_k^{i_k},\phi_\beta)\otimes \Omega_{g_2,n_2+1}^{\mathbf{t}}(\otimes_{k\in I_2}\phi_{\alpha_k}\psi_k^{i_k}\otimes \phi^\beta)
\end{align}
and
\begin{align}\label{eqn:split-ax-2}
\iota_{g-1}^*\Omega_{g,n}^{\mathbf{t}}(\phi_{\alpha_1}\psi_1^{i_1},\dots,\phi_{\alpha_n}\psi_n^{i_n})
=\sum_{\beta}\Omega^{\mathbf{t}}_{g-1,n+2}(\phi_{\alpha_1}\psi_1^{i_1},\dots,\phi_{\alpha_n}\psi_n^{i_n},\phi_\beta,\phi^\beta)
\end{align}
where $\iota_{g_1,g_2}\colon \overline{\mathcal{M}}_{g_1,n_1+1}\times \overline{\mathcal{M}}_{g_2,n_2+1}\rightarrow \overline{\mathcal{M}}_{g,n}$ and $\iota_{g-1}\colon \overline{\mathcal{M}}_{g-1,n+2}\rightarrow \overline{\mathcal{M}}_{g,n}$ are the two canonical gluing maps.
Note that, by definition, $\Omega$ is related to the double bracket of descendant potential via integration on the moduli space of curve, 
i.e.
\begin{align*}
\int_{\overline{\mathcal{M}}_{g,n}}\Omega_{g,n}^{\mathbf{t}}(v_1\psi_1^{i_1},\dots,v_n\psi_{n}^{i_n})
=\langle\langle\tau_{i_1}(v_1),\dots,\tau_{i_n}(v_n)\rangle\rangle_g(\mathbf{t}).
\end{align*}

\subsection{Proof of Theorem~\ref{thm:new-type-uni-constr-descen-gw}}
\label{subsec:proof-of-thm1.5}
In this subsection, we explain how to use Pixton's formula to rewrite each  factor of the form $\langle\langle \mathcal{W}_1\ldots\mathcal{W}_k;\lambda_{g}\rangle\rangle_g$ (with $k=1, 2$) in the right hand side of equation~\eqref{eqn:Theta-formula} in terms of pure descendant Gromov-Witten invariants.
For convenience, let $\mathsf{P}^g_{\Gamma}(0,...,0)$ denote the constant term of the following polynomial in $r$ for sufficiently large $r$
\begin{align}\label{eqn:PgGamma00}
\mathsf{P}^g_{\Gamma}(0,...,0):=\text{Coeff}_{r^{0}}
\left[\sum_{w\in \mathsf{W}_{\Gamma,r}}
\frac{1}{|\Aut(\Gamma)| } 
\frac{1}{r^{h^1(\Gamma)}}
\prod_{e=(h,h')\in E(\Gamma)}
\frac{1-\exp(-w(h)w(h')(\psi_h+\psi_{h'}))}{\psi_h + \psi_{h'}}\right]
\end{align}
for any $\Gamma\in G_{g,k}$.
We may view $\mathsf{P}^g_{\Gamma}(0,...,0)$ as a polynomial in the
$\psi$-classes $\{\psi_h:h\in H(\Gamma)\}$.

We next define a linear operator $P \mapsto \langle\langle \mathcal{W}_1\ldots\mathcal{W}_k; P\rangle\rangle_\Gamma$ that turns
any polynomial in the $\psi$-classes $\{\psi_h:h\in H(\Gamma)\}$ into a product of correlation functions.
Given a monomial $P = c \cdot \prod_{h\in H(\Gamma)} \psi_h^{a_h}$, we let
$\langle\langle \mathcal{W}_1\ldots\mathcal{W}_k; P\rangle\rangle_\Gamma$ to be the contraction of the multilinear form
\begin{equation*}
  c \prod_{v\in V(\Gamma)} \langle\langle\ldots\rangle\rangle_{g(v)}
\end{equation*}
with arguments corresponding to half-edges $h \in H(\Gamma)$ via:
\begin{itemize}
\item for the argument corresponding to the $i$-th leg, use the vector $\mathcal{W}_i$;
\item for each edge $e = (h,  h')$, use the bivector $T^{a_h}(\phi_\alpha)\otimes T^{a_{h'}}(\phi^\alpha)$ for the corresponding arguments.
\end{itemize}
\begin{proposition}\label{prop:TaTb}
For $k\geq1$,
\begin{align}\label{eqn:eqn:W-P=W-P-Gamma}
\langle\langle\mathcal{W}_1\ldots\mathcal{W}_k;\mathsf{P}_{g}^{g}(0)\rangle\rangle_g
=\sum_{\Gamma\in {G}_{g,k}} \langle\langle\mathcal{W}_1\ldots\mathcal{W}_k; \mathsf{P}^g_{\Gamma}(0)\rangle\rangle_\Gamma.
\end{align}
\end{proposition}

\begin{proof}
It is sufficient to prove equation~\eqref{eqn:eqn:W-P=W-P-Gamma} for $\mathcal{W}_i=T^{a_i}(v_i)$ for any $a_i\geq0$ and $v_i\in H^*(X;\mathbb{C})$.
 By equation~\eqref{eqn:anc-des-2} in Lemma~\ref{lem:anc-desc} and equation~\eqref{eqn:lambda-g-formula}, together with definition of $\Omega^{\mathbf{t}}$, we have
\begin{align*}
&\langle\langle T^{a_1}(v_1)\ldots T^{a_k}(v_k);\mathsf{P}_{g}^{g}(0,..,0)\rangle\rangle_g=\langle\langle \bar{\tau}_{a_1}(v_1)\ldots\bar{\tau}_{a_k}(v_k);\mathsf{P}_{g}^{g}(0,...,0)\rangle\rangle_g \\=&\sum_{m\geq0}\sum_{B\in H_2(X;\mathbb{Z})}\frac{q^B}{m!}
\int_{[\overline{\mathcal{M}}_{g,m+k}(X,B)]^{vir}}\left(\prod_{i=1}^{k}\bar{\psi}_{i}^{a_i}\ev_i^* v_i\right)\prod_{j=k+1}^{k+m}\ev_j^*\mathbf{t}(\psi_{j})\cdot \pi_m^* \St^*\mathsf{P}_{g}^{g}(0,...,0)
\\=&\int_{\overline{\mathcal{M}}_{g,k}}\prod_{i=1}^{k}\psi_i^{a_i}\cdot \Omega^{\mathbf{t}}_{g,k}(v_1,...,v_k)\cdot \mathsf{P}_g^g(0,...,0). 
\end{align*}
By the projection formula, this equals
\begin{align*}
&\sum_{\Gamma\in {G}_{g,k}}\int_{\overline{\mathcal{M}}_{g,k}}\prod_{i=1}^{k}\psi_i^{a_i}\cdot\Omega^{\mathbf{t}}_{g,k}(v_1,...,v_k)\cdot {\xi_{\Gamma}}_*\mathsf{P}^g_{\Gamma}(0,..,0)
\\=&\sum_{\Gamma\in {G}_{g,k}}\int_{\overline{\mathcal{M}}_{\Gamma}}\prod_{i=1}^{k}\psi_i^{a_i}\cdot {\xi_{\Gamma}}^*\Omega^{\mathbf{t}}_{g,k}(v_1,...,v_k)\cdot \mathsf{P}^g_{\Gamma}(0,...,0).
\end{align*}
Then by the splitting axioms~\eqref{eqn:split-ax-1}, \eqref{eqn:split-ax-2} and  equation~\eqref{eqn:PgGamma00}, we can express each summand as a product of descendant correlation functions according to the operator $P \mapsto \langle\langle T^{a_1}(v_1)\ldots T^{a_k}(v_k); P\rangle\rangle_\Gamma$.
In the process, the $k$ marked points are assigned the vectors $v_1,...,v_k$, and the cotangent line class $\psi$ is translated to descendants using the operator $T$.
Each node is translated into a sum of pairs of primary vector fields $\phi_\alpha$ and $\phi^\alpha$.
\end{proof}
Therefore, combining equation~\eqref{eqn:PgGamma00}, \eqref{eqn:eqn:W-P=W-P-Gamma} and Theorem~\ref{thm:virasoro-lambda-g-conj-rela}, we get Theorem~\ref{thm:new-type-uni-constr-descen-gw}.

\subsection{Application: $\lambda_{g}$ constraints in genus one and proof of Theorem~\ref{thm:lambda-g-conj-g=1}}
As an application of Theorem~\ref{thm:new-type-uni-constr-descen-gw}, we will compute $\Theta^{\mathsf{P}}_{1,n,m,\beta}$ explicitly and prove its vanishing. In fact, if genus $g=1$, Pixton DR class on $\overline{\mathcal{M}}_{1,k}$
\begin{align}
  \label{eqn:P_1^1(00)}
\mathsf{P}_{1}^{1}(0,..,0) =-\frac{1}{12} \xi_{\Gamma*}(1)  
\end{align}
where $\Gamma \in G_{1,k}$ is the dual graph of the boundary divisor of singular stable curves
with a non-separating node.

Then by equation~\eqref{eqn:theta-Pixton-class-gnmbeta} and the rule of translating double bracket $\langle\langle\,;\mathsf{P}\rangle\rangle_{\Gamma}$ to double bracket of descendant Gromov-Witten invariants in subsection~\ref{subsec:proof-of-thm1.5}, together with Proposition~\ref{prop:TaTb}, we obtain
\begin{align}\label{eqn:-12thetaPixton-1nmbeta}
&-12\Theta^{\mathsf{P}}_{1,n,m,\beta}
\nonumber\\=    
&\sum_{j=0}^{n+1}\sum_{r,\alpha,\sigma}e_{n+1-j}(r+b_{\alpha}-\frac{1}{2},...,r+n+b_{\alpha}-\frac{1}{2})\tilde{t}_r^{\alpha} 
\langle\langle\tau_{r+n-j}(c_1(X)^{j}\cup\phi_\alpha)\tau_{m}(\phi_\beta)\phi_\sigma\phi^\sigma\rangle\rangle^{X}_{0}
\nonumber\\&+
\sum_{j=0}^{n+1}\sum_{\sigma}e_{n+1-j}(m+b_{\beta}+\frac{1}{2},...,m+n+b_{\beta}+\frac{1}{2})\langle\langle \tau_{m+n-j}(c_1(X)^{j}\cup\phi_{\beta})\phi_\sigma\phi^\sigma\rangle\rangle^{X}_{0}
\nonumber\\&+\sum_{j=0}^{n+1}\sum_{s,\alpha,\sigma}(-1)^{-s-1}e_{n+1-j}(-s+b^{\alpha}-\frac{3}{2},...,-s+n+b^{\alpha}-\frac{3}{2})
\nonumber\\&\quad\quad\cdot\langle\langle\tau_{m}(\phi_\beta)\tau_{-s-1+n-j}(c_1(X)^j\cup\phi^\alpha)\rangle\rangle^{X}_{0}
\langle\langle\tau_{s}(\phi_\alpha)\phi_\sigma\phi^\sigma\rangle\rangle^{X}_{0}
\nonumber\\&+\sum_{j=0}^{n+1}\sum_{s,\alpha,\sigma}(-1)^{-s-1}e_{n+1-j}(-s+b^{\alpha}-\frac{3}{2},...,-s+n+b^{\alpha}-\frac{3}{2})
\nonumber\\&\quad\quad\cdot\langle\langle\tau_{m}(\phi_\beta)\tau_{-s-1+n-j}(c_1(X)^j\cup\phi^\alpha)\phi_\sigma\phi^\sigma\rangle\rangle^{X}_{0}
\langle\langle\tau_{s}(\phi_\alpha)\rangle\rangle^{X}_{0}.
\end{align}
By Theorem~\ref{thm:lambda-g-conj-g=0}, we have
\begin{align*}
&\Theta_{0,n,m,\beta}
\nonumber\\=    
&\sum_{j=0}^{n+1}\sum_{r,\alpha}e_{n+1-j}(r+b_{\alpha}-\frac{1}{2},...,r+n+b_{\alpha}-\frac{1}{2})\tilde{t}_r^{\alpha} 
\langle\langle\tau_{r+n-j}(c_1(X)^{j}\cup\phi_\alpha)\tau_{m}(\phi_\beta)\rangle\rangle^{X}_{0}
\nonumber\\&+
\sum_{j=0}^{n+1}e_{n+1-j}(m+b_{\beta}+\frac{1}{2},...,m+n+b_{\beta}+\frac{1}{2})\langle\langle \tau_{m+n-j}(c_1(X)^{j}\cup\phi_{\beta})\rangle\rangle^{X}_{0}
\nonumber\\&+\sum_{j=0}^{n+1}\sum_{s,\alpha}(-1)^{-s-1}e_{n+1-j}(-s+b^{\alpha}-\frac{3}{2},...,-s+n+b^{\alpha}-\frac{3}{2})
\nonumber\\&\quad\quad\cdot\langle\langle\tau_{m}(\phi_\beta)\tau_{-s-1+n-j}(c_1(X)^j\cup\phi^\alpha)\rangle\rangle^{X}_{0}
\langle\langle\tau_{s}(\phi_\alpha)\rangle\rangle^{X}_{0}
\nonumber\\&+\delta_{m}^{0}\sum_{\alpha}t_0^{\alpha}\int_{X}c_1(X)^{n+1}\cup\phi_\alpha\cup\phi_\beta
\\=&0.
\end{align*}
% Taking derivative along $\phi_\sigma$, we get
% \begin{align*}
% &\sum_{j=0}^{n+1}\sum_{r,\alpha}e_{n+1-j}(r+b_{\alpha}-\frac{1}{2},...,r+n+b_{\alpha}-\frac{1}{2})\tilde{t}_r^{\alpha} 
% \langle\langle\tau_{r+n-j}(c_1^{j}\cdot\phi_\alpha)\tau_{m}(\phi_\beta)\phi_\sigma\rangle\rangle^{X}_{0}
% \\&+\sum_{j=0}^{n+1}e_{n+1-j}(b_{\sigma}-\frac{1}{2},...,n+b_{\sigma}-\frac{1}{2}) 
% \langle\langle\tau_{n-j}(c_1^{j}\cdot\phi_\sigma)\tau_{m}(\phi_\beta)\rangle\rangle^{X}_{0}
% \nonumber\\&+
% \sum_{j=0}^{n+1}e_{n+1-j}(m+b_{\beta}+\frac{1}{2},...,m+n+b_{\beta}+\frac{1}{2})\langle\langle \tau_{m+n-j}(c_1^{j}\cdot\phi_{\beta})\phi_\sigma\rangle\rangle^{X}_{0}
% \nonumber\\&+\sum_{j=0}^{n+1}\sum_{s,\alpha}(-1)^{-s-1}e_{n+1-j}(-s+b^{\alpha}-\frac{3}{2},...,-s+n+b^{\alpha}-\frac{3}{2})
% \nonumber\\&\quad\quad\cdot\langle\langle\tau_{m}(\phi_\beta)\tau_{-s-1+n-j}(c_1^j\cdot\phi^\alpha)\phi_\sigma\rangle\rangle^{X}_{0}
% \langle\langle\tau_{s}(\phi_\alpha)\rangle\rangle^{X}_{0}
% \nonumber\\&+\sum_{j=0}^{n+1}\sum_{s,\alpha}(-1)^{-s-1}e_{n+1-j}(-s+b^{\alpha}-\frac{3}{2},...,-s+n+b^{\alpha}-\frac{3}{2})
% \nonumber\\&\quad\quad\cdot\langle\langle\tau_{m}(\phi_\beta)\tau_{-s-1+n-j}(c_1^j\cdot\phi^\alpha)\rangle\rangle^{X}_{0}
% \langle\langle\tau_{s}(\phi_\alpha)\phi_\sigma\rangle\rangle^{X}_{0}
% \nonumber\\&+\delta_{m}^{0}(\mathcal{C}^{n+1})_{\sigma\beta}
% \\=&0.
% \end{align*}
Taking second derivatives of the above equation along $\frac{\partial^2}{\partial t_0^\sigma\partial t_0^\epsilon}$,
then multiplying $\eta^{\sigma\epsilon}$ and taking sum over $\sigma$ and $\epsilon$, we get
\begin{align}\label{eqn:genus-0theta-second-der}
&\sum_{j=0}^{n+1}\sum_{r,\alpha,\sigma}e_{n+1-j}(r+b_{\alpha}-\frac{1}{2},...,r+n+b_{\alpha}-\frac{1}{2})\tilde{t}_r^{\alpha} 
\langle\langle\tau_{r+n-j}(c_1(X)^{j}\cup\phi_\alpha)\tau_{m}(\phi_\beta)\phi_\sigma\phi^\sigma\rangle\rangle^{X}_{0}
\nonumber\\&+2\sum_{j=0}^{n+1}\sum_{\sigma}e_{n+1-j}(b_{\sigma}-\frac{1}{2},...,n+b_{\sigma}-\frac{1}{2}) 
\langle\langle\tau_{n-j}(c_1(X)^{j}\cup\phi_\sigma)\phi^\sigma\tau_{m}(\phi_\beta)\rangle\rangle^{X}_{0}
\nonumber\\&+
\sum_{j=0}^{n+1}\sum_{\sigma}e_{n+1-j}(m+b_{\beta}+\frac{1}{2},...,m+n+b_{\beta}+\frac{1}{2})\langle\langle \tau_{m+n-j}(c_1(X)^{j}\cup\phi_{\beta})\phi_\sigma\phi^\sigma\rangle\rangle^{X}_{0}
\nonumber\\&+\sum_{j=0}^{n+1}\sum_{s,\alpha,\sigma}(-1)^{-s-1}e_{n+1-j}(-s+b^{\alpha}-\frac{3}{2},...,-s+n+b^{\alpha}-\frac{3}{2})
\nonumber\\&\quad\quad\cdot\langle\langle\tau_{m}(\phi_\beta)\tau_{-s-1+n-j}(c_1(X)^j\cup\phi^\alpha)\phi_\sigma\phi^\sigma\rangle\rangle^{X}_{0}
\langle\langle\tau_{s}(\phi_\alpha)\rangle\rangle^{X}_{0}
\nonumber\\&+2\sum_{j=0}^{n+1}\sum_{s,\alpha,\sigma}(-1)^{-s-1}e_{n+1-j}(-s+b^{\alpha}-\frac{3}{2},...,-s+n+b^{\alpha}-\frac{3}{2})
\nonumber\\&\quad\quad\cdot\langle\langle\tau_{m}(\phi_\beta)\tau_{-s-1+n-j}(c_1(X)^j\cup\phi^\alpha)\phi_\sigma\rangle\rangle^{X}_{0}
\langle\langle\phi^\sigma\tau_{s}(\phi_\alpha)\rangle\rangle^{X}_{0}
\nonumber\\&+\sum_{j=0}^{n+1}\sum_{s,\alpha,\sigma}(-1)^{-s-1}e_{n+1-j}(-s+b^{\alpha}-\frac{3}{2},...,-s+n+b^{\alpha}-\frac{3}{2})
\nonumber\\&\quad\quad\cdot\langle\langle\tau_{m}(\phi_\beta)\tau_{-s-1+n-j}(c_1(X)^j\cup\phi^\alpha)\rangle\rangle^{X}_{0}
\langle\langle\tau_{s}(\phi_\alpha)\phi_\sigma\phi^\sigma\rangle\rangle^{X}_{0}
\nonumber\\=&0.
\end{align}
Combining  equation~\eqref{eqn:-12thetaPixton-1nmbeta} and \eqref{eqn:genus-0theta-second-der}, 
we have
\begin{align}\label{eqn:6theta-g=0-bef-trr}
&6\Theta^{\mathsf{P}}_{1,n,m,\beta}
\nonumber\\=
&\sum_{j=0}^{n+1}\sum_{\sigma}e_{n+1-j}(b_{\sigma}-\frac{1}{2},...,n+b_{\sigma}-\frac{1}{2}) 
\langle\langle\tau_{m}(\phi_\beta)\tau_{n-j}(c_1(X)^{j}\cup\phi_\sigma)\phi^\sigma\rangle\rangle^{X}_{0}
\nonumber\\&+\sum_{j=0}^{n+1}\sum_{s,\alpha,\sigma}(-1)^{-s-1}e_{n+1-j}(-s+b^{\alpha}-\frac{3}{2},...,-s+n+b^{\alpha}-\frac{3}{2})
\nonumber\\&\quad\quad\cdot\langle\langle\tau_{m}(\phi_\beta)\tau_{-s-1+n-j}(c_1(X)^j\cup\phi^\alpha)\phi_\sigma\rangle\rangle^{X}_{0}
\langle\langle\phi^\sigma\tau_{s}(\phi_\alpha)\rangle\rangle^{X}_{0}.
\end{align}
Recall the genus-0 topological recursion relation has the form (cf. \cite{witten1990two}), 
\begin{align*}
\langle\langle\tau_{n+1}(\phi_\alpha)\tau_{m}(\phi_\beta)\phi_{k}(\phi_\gamma)\rangle\rangle_{0}   =\sum_{\sigma}\langle\langle\tau_{n}(\phi_\alpha)\phi_\sigma\rangle\rangle_{0}   \langle\langle\phi^\sigma\tau_{m}(\phi_\beta)\phi_{k}(\phi_\gamma)\rangle\rangle_{0}   
\end{align*}
for any $\alpha, \beta, \gamma$ and $n,m,k\geq0$.

Applying genus-0 topological recursion relation to the second term in the right hand side of equation~\eqref{eqn:6theta-g=0-bef-trr}, we get 
\begin{align*}
&6\Theta^{\mathsf{P}}_{1,n,m,\beta}
\\=&\sum_{j=0}^{n+1}\sum_{\sigma}e_{n+1-j}(b_{\sigma}-\frac{1}{2},...,n+b_{\sigma}-\frac{1}{2}) 
\langle\langle\tau_{m}(\phi_\beta)\tau_{n-j}(c_1(X)^{j}\cup\phi_\sigma)\phi^\sigma\rangle\rangle_{0}
\\&+\sum_{j=0}^{n+1}\sum_{s\geq0,\alpha}(-1)^{-s-1}e_{n+1-j}(-s+b^{\alpha}-\frac{3}{2},...,-s+n+b^{\alpha}-\frac{3}{2})
\nonumber\\&\quad\quad\cdot\langle\langle\tau_{m}(\phi_\beta)\tau_{-s-1+n-j}(c_1(X)^j\cup\phi^\alpha)\tau_{s+1}(\phi_\alpha)\rangle\rangle_{0}
\\=&
\sum_{j=0}^{n}\sum_{s,\alpha}(-1)^{-s}e_{n+1-j}(-s+b^{\alpha}-\frac{1}{2},...,-s+n+b^{\alpha}-\frac{1}{2})
\langle\langle\tau_{m}(\phi_\beta)\tau_{-s+n-j}(c_1(X)^j\cup\phi^\alpha)\tau_{s}(\phi_\alpha)\rangle\rangle_{0}.
\end{align*}
Then taking transformation of index $s=-s'+n-j$ and using the basic fact that $(\mathcal{C}^j)^{\alpha\sigma}\neq0$ implies $j+b^\alpha+b^\sigma=1$, we have
\begin{align*}
&\sum_{j=0}^{n}\sum_{s,\alpha}(-1)^{-s}e_{n+1-j}(-s+b^{\alpha}-\frac{1}{2},...,-s+n+b^{\alpha}-\frac{1}{2})
\langle\langle\tau_{m}(\phi_\beta)\tau_{-s+n-j}(c_1(X)^j\cup\phi^\alpha)\tau_{s}(\phi_\alpha)\rangle\rangle_{0} 
% \\=&\sum_{j,s'}\sum_{\alpha,\sigma}(-1)^{-s}e_{n+1-j}(s'+j-n+b^\alpha-\frac{1}{2},...,s'+j+b^\alpha-\frac{1}{2})
% \langle\langle\tau_m(\phi_\beta)\tau_{s'}(\phi_\sigma)\tau_{s}(\phi_\alpha)\rangle\rangle_{0}(\mathcal{C}^j)^{\alpha\sigma}
\\=&\sum_{j,s'}\sum_{\alpha,\sigma}(-1)^{-s'+n-j}e_{n+1-j}(s'-n+1-b^\sigma-\frac{1}{2},...,s'+1-b^\sigma-\frac{1}{2})
\\&\hspace{60pt}\cdot
\langle\langle\tau_m(\phi_\beta)\tau_{s'}(\phi_\sigma)\tau_{-s'+n-j}(\phi_\alpha)\rangle\rangle_{0}(\mathcal{C}^j)^{\alpha\sigma}
\\=&-\sum_{j,s}\sum_{\alpha,\sigma}(-1)^{-s}e_{n+1-j}(-s+n+b^\sigma-\frac{1}{2},...,-s+b^\sigma-\frac{1}{2})
\langle\langle\tau_m(\phi_\beta)\tau_{s}(\phi_\sigma)\tau_{-s+n-j}(c_1(X)^j\cup\phi^\sigma)\rangle\rangle_{0}
\\=&0.
\end{align*}
Thus we proved
\begin{align*}
\Theta^{\mathsf{P}}_{1,n,m,\beta}=0.    
\end{align*}
Therefore we finish the proof of Theorem~\ref{thm:lambda-g-conj-g=1}.

~\;

\noindent {\bf  Acknowledgements.}
 The author would  like to thank professor Xiaobo Liu,  Felix Janda, Yunfeng Jiang,  Feng Qu, Hao Xu, Di Yang and Youjin Zhang for discussions about  Virasoro  constraints  for Gromov-Witten invariants, Hodge integrals and Frobenius manifolds. 
 This work is supported by National Science Foundation of China (Grant No. 11601279)  and  Shandong Provincial Natural Science Foundation (Grant No. ZR2021MA101). 
 
% \hspace*{\fill}

% \noindent{\bf Data availability} Data sharing not applicable to this article as no datasets were generated or analysed.

% \hspace*{\fill}

% \noindent{\bf Declarations}

% \hspace*{\fill}

% \noindent{\bf Conflict of interest} The authors declared that they have no conflicts of interest to this work.
% \bibliographystyle{alpha}
% \bibliography{references}

\end{document}